\documentclass[11pt]{amsart}
\usepackage{amsmath,amssymb,amsthm,enumitem,url,tikz-cd,bbm}
\usepackage[hmargin=20mm,vmargin=30mm]{geometry}

\setlist[enumerate]{leftmargin=*}
\setlist[itemize]{labelindent=\parindent, leftmargin=*}

\numberwithin{equation}{section}

\theoremstyle{plain}
\newtheorem{thm}{Theorem}[section]
\newtheorem{lem}[thm]{Lemma}
\newtheorem{prop}[thm]{Proposition}

\theoremstyle{definition}

\theoremstyle{remark}
\newtheorem{rem}[thm]{Remark}
\newtheorem*{cav}{Caveat}

\newcommand\diag{\operatorname{diag}}
\newcommand\disc{\operatorname{disc}}
\newcommand\Hom{\operatorname{Hom}}
\newcommand\Ind{\operatorname{Ind}}
\newcommand\Irr{\operatorname{Irr}}
\renewcommand\Re{\operatorname{Re}}
\newcommand\Tr{\operatorname{Tr}}

\newcommand\down{\mathrm{down}}
\newcommand\GL{\mathrm{GL}}
\newcommand\Mp{\mathrm{Mp}}

\newcommand\SL{\mathrm{SL}}
\newcommand\Sp{\mathrm{Sp}}

\newcommand\U{\mathrm{U}}
\newcommand\up{\mathrm{up}}

\newcommand\A{\mathbb{A}}
\newcommand\C{\mathbb{C}}
\newcommand\E{\mathbb{E}}
\newcommand\F{\mathbb{F}}
\newcommand\Q{\mathbb{Q}}
\newcommand\R{\mathbb{R}}
\newcommand\V{\mathbb{V}}
\newcommand\W{\mathbb{W}}
\newcommand\X{\mathbb{X}}
\newcommand\Y{\mathbb{Y}}
\newcommand\Z{\mathbb{Z}}

\newcommand\g{\mathfrak{g}}
\renewcommand\k{\mathfrak{k}}
\renewcommand\l{\mathfrak{l}}

\newcommand\q{\mathfrak{q}}
\renewcommand\t{\mathfrak{t}}
\renewcommand\u{\mathfrak{u}}

\newcommand\CC{\mathcal{C}}
\newcommand\FF{\mathcal{F}}

\newcommand\VV{\mathcal{V}}
\newcommand\XX{\mathcal{X}}
\newcommand\ZZ{\mathcal{Z}}

\newcommand\HH{\mathfrak{H}}

\newcommand\1{\mathbf{1}}
\newcommand\pp{\mathbf{p}}
\newcommand\qq{\mathbf{q}}
\newcommand\rr{\mathbf{r}}
\renewcommand\ss{\mathbf{s}}
\newcommand\xx{\mathbf{x}}

\newcommand\longhookrightarrow{\lhook\joinrel\longrightarrow}

\DeclareMathOperator*{\Res}{Res}

\title{Theta lifting for discrete series representations of real unitary groups}
\author{Atsushi Ichino}
\address{Department of Mathematics, Kyoto University, Kitashirakawa Oiwake-cho, Sakyo-ku, Kyoto 606-8502, Japan}
\email{ichino@math.kyoto-u.ac.jp}

\begin{document}

\maketitle

\begin{abstract}
We study the theta lifting for real unitary groups and completely determine the theta lifts of discrete series representations.
In particular, we show that these theta lifts can be expressed as cohomologically induced representations in the weakly fair range.
This extends a result of J.-S.~Li in the case of discrete series representations with sufficiently regular infinitesimal character, whose theta lifts can be expressed as cohomologically induced representations in the good range.
\end{abstract}

\section{Introduction}

In his seminal papers \cite{howe1, howe2}, Howe introduced the notion of reductive dual pairs and developed the theory of theta lifting, which has been an important subject in the representation theory of real and $p$-adic reductive groups for more than $40$ years and which has many arithmetic applications to the theory of automorphic forms.
The theta lifting is defined as a correspondence between representations of the two groups in a reductive dual pair in terms of the restriction of the Weil representation \cite{weil1}.
In fact, it is shown that this correspondence is one-to-one by Howe himself \cite{howe2} in the real case and by Gan--Takeda \cite{gt} in the $p$-adic case, following earlier work of Howe \cite{howe1} and Waldspurger \cite{wal} for $p \ne 2$.
For the history and recent development of the theta lifting, the reader can consult the ICM report of Gan \cite{gan}.

In the theory of theta lifting, one of the basic problems is to describe it explicitly.
We consider this problem in the real case, which has been studied by M{\oe}glin \cite{moeglin89}, Li \cite{li90}, Adams--Barbasch \cite{ab1, ab2}, Paul \cite{paul1, paul2, paul3}, Li--Paul--Tan--Zhu \cite{lptz} to mention a few, but which has not been solved in general.
For example, consider the reductive dual pair $(\U(p,q),\U(r,s))$ consisting of real unitary groups.
Then Paul \cite{paul1, paul2} completely determined the theta lifts when $p+q=r+s$ or $p+q=r+s \pm 1$.
However, beyond this case, it seems to be notoriously difficult to determine the theta lifts, especially when singular representations occur.
In fact, there has been no significant progress in this direction for almost $20$ years.

In this paper, we take a first step toward determining the theta lifts for real unitary groups.
More precisely, let $\pi$ be a discrete series representation of $\U(p,q)$ and consider its theta lift $\theta_{r,s}(\pi)$ to $\U(r,s)$ when $p+q < r+s$.
Then by a result of Li \cite{li90}, we have
\[
 \theta_{r,s}(\pi) = A_\q(\lambda)
\]
if the infinitesimal character of $\pi$ is sufficiently regular, where $A_\q(\lambda)$ is an explicit cohomologically induced representation in the good range.
The purpose of this paper is to remove this assumption on the infinitesimal character of $\pi$.
Namely, our main result (Theorem \ref{t:main}) roughly says the following.

\begin{thm}
Let $\pi$ be a discrete series representation of $\U(p,q)$.
Assume that its theta lift $\theta_{r,s}(\pi)$ to $\U(r,s)$ is nonzero.
\begin{enumerate}
\item 
If $p+q < r+s$, then $\theta_{r,s}(\pi)$ is a cohomologically induced representation $A_\q(\lambda)$ of $\U(r,s)$ in the weakly fair range, where $\q$ and $\lambda$ can be described explicitly.
\item 
If $p+q \ge r+s$, then $\theta_{r,s}(\pi)$ is a discrete series representation of $\U(r,s)$, where its Harish-Chandra parameter can be described explicitly.
\end{enumerate}
\end{thm}

We have stated the result under the assumption that $\theta_{r,s}(\pi)$ is nonzero, but there is a combinatorial criterion for the nonvanishing of $\theta_{r,s}(\pi)$ due to Atobe \cite{atobe} (see \S \ref{ss:nonvanishing} below).
Based on this theorem, we hope to describe $\theta_{r,s}(\pi)$ explicitly for more general $\pi$ in future work.

This paper is organized as follows.
In \S \ref{s:local-theta}, we review the notion of local theta lifting.
In \S \ref{s:real-rep}, we recall some representations of real unitary groups.
In \S \ref{s:main}, we state the main theorem of this paper.
To explain the idea of the proof, we should note the following.

\begin{cav}
Our proof is global and relies on Arthur's endoscopic classification \cite{arthur,mok,kmsw}.
Namely, our main result is conditional on Arthur's multiplicity formula for the automorphic discrete spectra of unitary groups announced by Kaletha--M\'inguez--Shin--White \cite{kmsw} (see \eqref{eq:amf} below for details), whose proof will be completed in their subsequent work.
\end{cav}

In \S \ref{s:packets}, we describe the representations in some local $L$- and $A$-packets for unitary groups explicitly.
This will be the input and output of Arthur's multiplicity formula.
In particular, a result of M{\oe}glin--Renard \cite{mr19} plays a crucial role in the proof since it expresses the representations in $A$-packets for real unitary groups in terms of cohomologically induced representations.
In \S \ref{s:proof}, we prove the main theorem.
We first globalize the given local theta lift for real unitary groups.
More precisely, we find a global theta lift such that
\begin{itemize}
\item at one real place, its localization is the theta lift of an arbitrary discrete series representation;
\item at another real place, its localization is the theta lift of a discrete series representation with sufficiently regular infinitesimal character, which is determined explicitly by Li \cite{li90};
\item at the other places, its localizations are easy to describe explicitly.
\end{itemize}
Then we use Arthur's multiplicity formula (viewed as a product formula) to transfer the information from the case of sufficiently regular infinitesimal character to the general case.
However, there is a serious technical difficulty in this argument: it is not straightforward to globalize a local theta lift for real unitary groups.

In \S 7, we overcome this difficulty, which we now explain in detail.
Let $\pi$ be a discrete series representation of $\U(p,q)$ and consider its theta lift $\theta_{r,s}(\pi)$ to $\U(r,s)$ when $p+q < r+s$.
Let $F \ne \Q$ be a totally real number field with ad\`ele ring of $\A$ and fix a real place $v_0$ of $F$.
Then it is easy to find
\begin{itemize}
\item anisotropic unitary groups $G$ and $H$ over $F$ such that $G_{v_0} = \U(p,q)$ and $H_{v_0} = \U(r,s)$, respectively;
\item an irreducible automorphic representation of $G(\A)$ such that $\varPi_{v_0} = \pi$.
\end{itemize}
But we need $G,H,\varPi$ such that the global theta lift $\theta(\varPi)$ to $H(\A)$ is nonzero.
For this, we proceed as follows.
\begin{enumerate}
\item 
\label{item:intro1}
Find $G,H,\varPi$ such that the local theta lift $\theta(\varPi_v)$ to $H_v$ is nonzero for all places $v$ of $F$.
\item 
\label{item:intro2}
Show that $\theta(\varPi)$ is nonzero if and only if $\theta(\varPi_v)$ is nonzero for all $v$.
\end{enumerate}
To show that $G,H,\varPi$ as in \eqref{item:intro1} exist, we appeal to Arthur's multiplicity formula.
In fact, we may impose further local conditions on $G,H,\varPi$ to make the global-to-local argument work.
On the other hand, \eqref{item:intro2} is largely but not completely known for unitary groups (see \cite[Theorem 1.3]{gqt}).
Indeed, the standard argument relies on the Rallis inner product formula, which contains the local integral at $v_0$ given by
\[
 \int_{\U(p,q)} (\omega(g) \varphi_1, \varphi_2) \overline{(\pi(g)f_1, f_2)} \, dg
\]
for $\varphi_1, \varphi_2 \in \omega$ and $f_1, f_2 \in \pi$.
Here $\omega$ is the Weil representation of $\U(p,q) \times \U(r,s)$ and $(\cdot, \cdot)$ denotes an invariant Hermitian inner product.
This integral is absolutely convergent and defines an invariant functional
\[
 \ZZ_{r,s}(\pi) : \omega \otimes \bar{\omega} \otimes \bar{\pi} \otimes \pi \longrightarrow \C.
\]
Then we have $\theta_{r,s}(\pi) \ne 0$ if $\ZZ_{r,s}(\pi) \ne 0$, and we are reduced to prove the converse.
However, it was previously only known that if $\theta_{r,s}(\pi) \ne 0$, then $\ZZ_{r', s'}(\pi) \ne 0$ for some $r', s'$ such that $r'+s'=r+s$ and $r' \equiv r \bmod 2$ (see \cite[Proposition 11.5]{gqt}).
Thus we need to prove the following (see Proposition \ref{p:key} below), which is the key technical innovation in this paper.

\begin{prop}
Let $\pi$ be a discrete series representation of $\U(p,q)$.
Then we have
\[
 \theta_{r,s}(\pi) \ne 0 \quad \Longleftrightarrow \quad \ZZ_{r,s}(\pi) \ne 0.
\]
\end{prop}

To prove this proposition, we modify the argument of Atobe \cite{atobe} for the nonvanishing of $\theta_{r,s}(\pi)$, which uses the Gan--Gross--Prasad conjecture partially proved by He \cite{he}.
We stress that the proof is local and does not rely on Arthur's endoscopic classification.
In particular, the result and its application to the nonvanishing of global theta lifts are unconditional.

\subsection*{Acknowledgements}

We would like to thank Hiraku Atobe for useful discussions.
The author is partially supported by JSPS KAKENHI Grant Number 19H01781.

\subsection*{Notation}

For any representation $\pi$, we denote by $\pi^\vee$ the contragredient of $\pi$ and by $\bar{\pi}$ the complex conjugate of $\pi$.
For any real reductive group $G$, we work with the category of $(\g,K)$-modules unless otherwise specified, where $\g$ is the complexified Lie algebra of $G$ and $K$ is a maximal compact subgroup of $G$.
Thus by abuse of terminology, we usually mean a $(\g,K)$-module by a representation of $G$.

\section{Local theta lifting}
\label{s:local-theta}

In this section, we review the notion of local theta lifting.
We follow the convention in \cite{gi1, gi2}, which is different from that in \cite{kudla,hks}.

\subsection{Hermitian and skew-Hermitian spaces}
\label{ss:herm-spaces}

Let $F$ be a local field of characteristic zero.
Let $E$ be an \'etale quadratic algebra over $F$, so that $E$ is either $F \times F$ or a quadratic extension of $F$.
We denote by $c$ the nontrivial automorphism of $E$ over $F$.
Let $\Tr_{E/F}$ and $\operatorname{N}_{E/F}$ be the trace and norm maps from $E$ to $F$, respectively.
Let $\omega_{E/F}$ be the (possibly trivial) quadratic character of $F^\times$ associated to $E/F$ by local class field theory, so that $\operatorname{Ker}(\omega_{E/F}) = \operatorname{N}_{E/F}(E^\times)$.
Fix an element $\delta \in E^\times$ such that $\Tr_{E/F}(\delta) = 0$.

Fix $\varepsilon = \pm 1$.
Let $V$ be an $n$-dimensional $\varepsilon$-Hermitian space over $E$.
Namely, $V$ is a free $E$-module of rank $n$ equipped with a nondegenerate sesquilinear form $\langle \cdot, \cdot \rangle_V : V \times V \rightarrow E$ satisfying
\[
 \langle a v, b w \rangle_V = a b^c \langle v, w \rangle_V, \quad
 \langle w, v \rangle_V = \varepsilon \langle v, w \rangle_V^c
\]
for $a, b \in E$ and $v, w \in V$.
Let $\det(V) \in E^\times / \operatorname{N}_{E/F}(E^\times)$ be the determinant of the matrix
\[
 (\langle v_i, v_j \rangle_V)_{1 \le i,j \le n},
\]
where $v_1, \dots, v_n$ is a basis of $V$.
We define $\epsilon(V) = \pm 1$ by
\[
 \epsilon(V) =
 \begin{cases}
  \omega_{E/F}((-1)^{\frac{1}{2}n(n-1)} \cdot \det(V)) & \text{if $\varepsilon = +1$;} \\
  \omega_{E/F}((-1)^{\frac{1}{2}n(n-1)} \cdot \det(V) \cdot \delta^{-n}) & \text{if $\varepsilon = -1$.}
 \end{cases}
\]
Note that $\epsilon(V)$ depends on $\delta$ if $\varepsilon = -1$, $E \ne F \times F$, and $n$ is odd.
We denote by $\U(V)$ the unitary group of $V$, i.e.
\[
 \U(V) = \{ g \in \GL(V) \, | \, \text{$\langle gv, gw \rangle_V = \langle v, w \rangle_V$ for all $v, w \in V$} \}.
\]

Recall that given a positive integer $n$, the $n$-dimensional $\varepsilon$-Hermitian spaces over $E$ (up to isometry) are classified as follows.
\begin{itemize}
\item 
If $E = F \times F$, then there is a unique such space.
We denote it by $V_n^+$.
Then we have $\epsilon(V_n^+) = +1$ and $V_n^+ = \VV_n \otimes_F E$ for some $n$-dimensional vector space $\VV_n$ over $F$.
Moreover, the first projection $V_n^+ = \VV_n \times \VV_n \rightarrow \VV_n$ induces an isomorphism $\U(V_n^+) \cong \GL(\VV_n)$.
\item 
If $F$ is nonarchimedean and $E \ne F \times F$, then there are precisely two such spaces, which are distinguished by their signs.
We denote them by $V_n^+$ and $V_n^-$ so that $\epsilon(V_n^+) = +1$ and $\epsilon(V_n^-) = -1$.
\item 
If $F = \R$ and $E = \C$, then there are precisely $n+1$ such spaces, which are distinguished by their signatures.
We denote by $V_{p,q}$ the space of signature $(p,q)$, where $p,q$ are nonnegative integers such that $p+q=n$.
More precisely, we require that $V_{p,q}$ has a basis $v_1, \dots, v_n$ such that
\[
 \langle v_i, v_j \rangle_{V_{p,q}} = \zeta \times 
 \begin{cases}
  1 & \text{if $i = j \le p$;} \\
  -1 & \text{if $i = j > p$;} \\
  0 & \text{if $i \ne j$,}
 \end{cases}
\]
where 
\[
 \zeta = 
 \begin{cases}
  1 & \text{if $\varepsilon = +1$;} \\
  \sqrt{-1} & \text{if $\varepsilon = -1$.}
 \end{cases}
\]
Then we have
\[
 \epsilon(V_{p,q}) = (-1)^{\frac{1}{2}(p-q)(p-q-1)}
\]
if we take $\delta = \sqrt{-1}$.
\end{itemize}

\subsection{Theta lifts}

Let $V$ be an $m$-dimensional Hermitian space over $E$ and $W$ an $n$-dimensional skew-Hermitian space over $E$.
We regard $\W = V \otimes_E W$ as a vector space over $F$ and equip it with the symplectic form given by
\[
 \langle \hspace{-1mm} \langle v_1 \otimes w_1, v_2 \otimes w_2 \rangle \hspace{-1mm} \rangle = \Tr_{E/F}(\langle v_1, v_2 \rangle_V \langle w_1, w_2 \rangle_W).
\]
Let $\Sp(\W)$ be the symplectic group of $\W$ and $\Mp(\W)$ the metaplectic $\C^1$-cover of $\Sp(\W)$.
Then it follows from \cite{kudla,hks} that the natural homomorphism $\U(V) \times \U(W) \rightarrow \Sp(\W)$ has a lift
\[
 \iota_{V,W,\chi_V,\chi_W,\psi} : \U(V) \times \U(W) \longrightarrow \Mp(\W)
\]
which depends on the choice of the following data:
\begin{itemize}
\item two unitary characters $\chi_V, \chi_W$ of $E^\times$ such that
\[
 \chi_V|_{F^\times} = \omega_{E/F}^m, \quad
 \chi_W|_{F^\times} = \omega_{E/F}^n;
\]
\item a nontrivial additive character $\psi$ of $F$.
\end{itemize}
Composing this with the Weil representation $\omega_\psi$ of $\Mp(\W)$ relative to $\psi$, we obtain a representation
\[
 \omega_{V, W, \chi_V, \chi_W, \psi} = \omega_\psi \circ \iota_{V, W, \chi_V, \chi_W, \psi}
\]
of $\U(V) \times \U(W)$.
Note that if we apply the construction to the spaces $W$ and $V$ equipped with the Hermitian form $\delta^{-1} \langle \cdot, \cdot \rangle_W$ and the skew-Hermitian form $\delta \langle \cdot, \cdot \rangle_V$, respectively, then we obtain the representation
\[
 \omega_{V, W, \chi_V, \chi_W, \psi} \circ \mathrm{sw},
\]
where $\mathrm{sw}: \U(W) \times \U(V) \rightarrow \U(V) \times \U(W)$ switches factors.
In particular, we can freely switch the roles of $V$ and $W$.

For any irreducible representation $\pi$ of $\U(W)$, the maximal $\pi$-isotypic quotient of $\omega_{V, W, \chi_V, \chi_W, \psi}$ is of the form
\[
 \Theta_{V,W,\chi_V,\chi_W,\psi}(\pi) \boxtimes \pi
\]
for some representation $\Theta_{V,W,\chi_V,\chi_W,\psi}(\pi)$ of $\U(V)$ of finite length.
We denote by $\theta_{V, W, \chi_V,\chi_W,\psi}(\pi)$ the maximal semisimple quotient of $\Theta_{V, W, \chi_V,\chi_W,\psi}(\pi)$ and call it the theta lift of $\pi$ to $\U(V)$.
By the Howe duality \cite{howe2,wal,minguez,gt}, $\theta_{V, W, \chi_V,\chi_W,\psi}(\pi)$ is either zero or irreducible.

\begin{lem}
\label{l:complex-conjugate-theta}
We have
\[
 \overline{\theta_{V,W,\chi_V,\chi_W,\psi}(\pi)} = \theta_{-V,W,\chi_V,\chi_W,\psi}(\bar{\pi} \otimes (\chi_V \circ \det)) \otimes (\chi_W^{-1} \circ \det),
\]
where $-V$ denotes the space $V$ equipped with the Hermitian form $- \langle \cdot, \cdot \rangle_V$.
\end{lem}

\begin{proof}
The assertion follows from the fact that
\begin{align*}
  \overline{\omega_{V,W,\chi_V,\chi_W,\psi}}
  & = \omega_{V,W,\chi_V^{-1},\chi_W^{-1},\psi^{-1}} \\
  & = \omega_{-V,W,\chi_V,\chi_W,\psi} \otimes 
  ((\chi_W^{-1} \circ \det) \boxtimes (\chi_V^{-1} \circ \det)).
\end{align*}
\end{proof}

\section{Representations of real unitary groups}
\label{s:real-rep}

In this section, we recall some representations of real unitary groups which we will use later.

\subsection{Discrete series representations}
\label{ss:ds}

Let $p, q$ be nonnegative integers and put $n=p+q$.
Let $G = \U(p,q)$ be the unitary group of signature $(p,q)$, which we realize as
\[
 \U(p,q) = \left\{ g \in \GL_n(\C) \, \left| \, {}^t \bar{g}
 \begin{pmatrix}
  \1_p & \\
  & -\1_q
 \end{pmatrix}
 g = 
 \begin{pmatrix}
  \1_p & \\
  & -\1_q 
 \end{pmatrix}
 \right. \right\}.
\]
We define a maximal compact subgroup $K \cong \U(p) \times \U(q)$ of $G$ by
\[
 K = \{ g \in G \, | \, {}^t \bar{g}^{-1} = g \}.
\]
Let $\g_0$ be the Lie algebra of $G$ and $\t_0$ the Cartan subalgebra of $\g_0$ consisting of diagonal matrices.
Let $\g = \g_0 \otimes_\R \C$ and $\t = \t_0 \otimes_\R \C$ be their complexifications.
We identify $\t^*$ with $\C^n$ via the basis $\varepsilon_1,\dots,\varepsilon_n$ given by
\[
 \varepsilon_i(\diag(a_1, \dots, a_n)) = a_i
\]
and define a bilinear form $\langle \cdot, \cdot \rangle : \t^* \times \t^* \rightarrow \C$ by
\[
 \langle x, y \rangle = x_1 y_1 + \dots + x_n y_n
\]
for $x = (x_1,\dots,x_n), y = (y_1,\dots,y_n) \in \t^* \cong \C^n$.
Let $\Delta$ be the set of roots of $\t$ in $\g$, so that 
\[
 \Delta = \{ \pm(\varepsilon_i - \varepsilon_j) \, | \, 1 \le i < j \le n \}.
\]
Let $\Delta_c$ be the set of compact roots in $\Delta$ and take the positive system $\Delta_c^+$ of $\Delta_c$ given by 
\[
 \Delta_c^+ = \{ \varepsilon_i - \varepsilon_j  \, | \, 1 \le i < j \le p \} \cup \{ \varepsilon_i - \varepsilon_j \, | \, p < i < j \le n \}.
\]
Then the discrete series representations of $G$ are parametrized by Harish-Chandra parameters (which are dominant for $\Delta_c^+$)
\[
 \lambda = (\lambda_1, \dots, \lambda_n) \in \sqrt{-1} \t_0^*,
\]
where 
\begin{itemize}
\item $\lambda_i \in \Z + \frac{n-1}{2}$;
\item $\lambda_i \ne \lambda_j$ if $i \ne j$;
\item $\lambda_1 > \dots > \lambda_p$ and $\lambda_{p+1} > \dots > \lambda_n$.
\end{itemize}

\subsection{Cohomologically induced representations}
\label{ss:cohom}

We retain the notation of the previous subsection.
In particular, $G = \U(p,q)$.
Let $\xx \in \sqrt{-1} \t_0$, so that the adjoint action $\operatorname{ad}(\xx)$ on $\g$ is diagonizable with real eigenvalues.
We denote by $\l(\xx)$ the sum of zero eigenspaces of $\operatorname{ad}(\xx)$ in $\g$, i.e.~the centralizer of $\xx$ in $\g$, and by $\u(\xx)$ the sum of positive eigenspaces of $\operatorname{ad}(\xx)$ in $\g$.
Then
\[
 \q(\xx) = \l(\xx) \oplus \u(\xx)
\]
is a $\theta$-stable parabolic subalgebra of $\g$.
We write $\q(\xx) = \q_{\pp,\qq}$ and $\l(\xx) = \l_{\pp,\qq}$ if $\xx$ is of the form
\[
 \xx = \diag(\underbrace{x_1,\dots,x_1}_{p_1}, \dots, \underbrace{x_k,\dots,x_k}_{p_k}, \underbrace{x_1,\dots,x_1}_{q_1}, \dots, \underbrace{x_k,\dots,x_k}_{q_k}),
\]
where $\pp = (p_1, \dots, p_k), \qq = (q_1, \dots, q_k)$ are tuples of nonnegative integers such that 
\[
 p_1 + \dots + p_k = p, \quad
 q_1 + \dots + q_k = q
\]
and $x_1, \dots, x_k$ are real numbers such that $x_1 > \dots > x_k$.
Let $L_{\pp,\qq}$ be the normalizer of $\q_{\pp,\qq}$ in $G$, so that
\[
 L_{\pp,\qq} \cong \U(p_1,q_1) \times \dots \times \U(p_k,q_k).
\]

For $\q = \q_{\pp,\qq}$ and 
\[
 \lambda = (\underbrace{\lambda_1,\dots,\lambda_1}_{p_1}, \dots, \underbrace{\lambda_k,\dots,\lambda_k}_{p_k}, \underbrace{\lambda_1,\dots,\lambda_1}_{q_1}, \dots, \underbrace{\lambda_k,\dots,\lambda_k}_{q_k}) \in \sqrt{-1} \t_0^*
\]
with $\lambda_i \in \Z$, which is the differential of the $1$-dimensional representation
\[
 {\det}^{\lambda_1} \boxtimes \cdots \boxtimes {\det}^{\lambda_k}
\]
of $L_{\pp,\qq}$, we consider the cohomologically induced representation
\[
 A_\q(\lambda)
\]
defined by \cite[(5.6)]{kv}.
The following summarizes some properties of $A_\q(\lambda)$.
\begin{itemize}
\item
If $\lambda$ is in the good range, i.e.~$\lambda_i - \lambda_{i+1} > -1$ for all $i$, then $A_\q(\lambda)$ is nonzero and irreducible.
\item 
If $\lambda$ is in the weakly fair range, i.e.~$\lambda_i - \lambda_{i+1} \ge -\frac{1}{2}(p_i+q_i+p_{i+1}+q_{i+1})$ for all $i$, then $A_\q(\lambda)$ is unitary (but possibly zero).
If further $A_\q(\lambda)$ is nonzero, then it is irreducible, which is a special property in the case of unitary groups (see \cite{matumoto, trapa}).
\item
If $\lambda$ is in the weakly fair range, then there is an algorithm due to Trapa \cite{trapa} which determines the nonvanishing and the Langlands parameter of $A_\q(\lambda)$ in the case of unitary groups.
\end{itemize}

\section{Statement of the main theorem}
\label{s:main}

In this section, we state the main theorem of this paper, which describes the theta lifts of discrete series representations of real unitary groups explicitly.

\subsection{Setup}
\label{ss:main-theorem-setup}

We consider the theta lifting from $\U(W)$ to $\U(V)$, where $W$ is an $n$-dimensional skew-Hermitian space over $\C$ and $V$ is an $m$-dimensional Hermitian space over $\C$.
Let $(p,q)$ and $(r,s)$ be the signatures of $W$ and $V$, respectively, so that $p+q = n$ and $r+s = m$. 
We identify $\U(W)$ and $\U(V)$ with $\U(p,q)$ and $\U(r,s)$, respectively, via the bases as in \S \ref{ss:herm-spaces}.

From now on, we take the characters $\chi_V, \chi_W$ of $\C^\times$ given by
\[
 \chi_V(z) = \left( \frac{z}{\sqrt{z \bar{z}}} \right)^{m_0}, \quad
 \chi_W(z) = \left( \frac{z}{\sqrt{z \bar{z}}} \right)^{n_0} 
\]
with some fixed integers $m_0, n_0$ such that
\[
 m_0 \equiv m \bmod 2, \quad
 n_0 \equiv n \bmod 2,
\]
and the character $\psi$ of $\R$ given by
\[
 \psi(x) = e^{-2 \pi \sqrt{-1} x}.
\]
(We make this choice so that Lemma \ref{l:K-type-corresp} below holds.)
Then we write the theta lift of an irreducible representation $\pi$ of $\U(W) = \U(p,q)$ to $\U(V) = \U(r,s)$ as
\[
 \theta_{r,s}(\pi) = \theta_{V, W, \chi_V, \chi_W, \psi}(\pi).
\]

\subsection{Explicit description of theta lifts}

We now state our main result.

\begin{thm}
\label{t:main}
Let $\pi$ be a discrete series representation of $\U(p,q)$ with Harish-Chandra parameter $\lambda$.
Assume that its theta lift $\theta_{r,s}(\pi)$ to $\U(r,s)$ is nonzero.
Put $n=p+q$ and $m=r+s$.
\begin{enumerate}
\item 
\label{item:main1}
If $m > n$, then we have
\[
 \theta_{r,s}(\pi) = A_\q(\lambda'), 
\]
where $\q$ and $\lambda'$ are given as follows.
We may write
\[
 \lambda = (\alpha_1, \dots, \alpha_x, \beta_1, \dots, \beta_y, \gamma_1, \dots, \gamma_z, \delta_1, \dots, \delta_w) + \bigg( \frac{m_0}{2}, \dots, \frac{m_0}{2} \bigg)
\]
with
\begin{itemize}
\item $\alpha_i, \gamma_j > 0$ and $\beta_i, \delta_j \le 0$;
\item $x+y = p$ and $z+w = q$;
\item $x+w \le r$ and $z+y \le s$.
\end{itemize}
Then $\q = \q(\xx)$ is associated to
\[
 \xx = (\alpha_1, \dots, \alpha_x, \underbrace{\epsilon, \dots, \epsilon}_{r-x-w}, \delta_1, \dots, \delta_w, \gamma_1, \dots, \gamma_z, \underbrace{\epsilon, \dots, \epsilon}_{s-z-y}, \beta_1, \dots, \beta_y)
\]
with $\epsilon \in \R$ such that $\min\{ \alpha_x, \gamma_z\} > \epsilon > \max \{ \beta_1, \delta_1 \}$ and $\lambda'$ is given by
\[
 \lambda' = (\alpha_1', \dots, \alpha_x', \epsilon', \dots, \epsilon', \delta_1', \dots, \delta_w', \gamma_1', \dots, \gamma_z', \epsilon', \dots, \epsilon', \beta_1', \dots, \beta_y')
 + \bigg(\frac{n_0}{2}, \dots, \frac{n_0}{2}\bigg)
\]
with 
\begin{align*}
 \alpha_i' & = \alpha_i - \frac{m+1}{2} + i + \# \{ k \, | \, \gamma_k > \alpha_i \}, \\ 
 \beta_j' & = \beta_j + \frac{m-1}{2} + j-y - \# \{ l \, | \, \delta_l < \beta_j \}, \\
 \gamma_k' & = \gamma_k -\frac{m+1}{2} + k + \# \{ i \, | \, \alpha_i > \gamma_k \}, \\
 \delta_l' & = \delta_l + \frac{m-1}{2} + l-w - \# \{ j \, | \, \beta_j < \delta_l \}, \\
 \epsilon' & = x+z - \frac{n}{2}.
\end{align*}
\item 
\label{item:main2}
If $m \le n$, then $\theta_{r,s}(\pi)$ is the discrete series representation of $\U(r,s)$ with Harish-Chandra parameter $\lambda'$, where $\lambda'$ is given as follows.
Put $k=n-m$.
We may write either
\[
 \lambda = \bigg(\alpha_1, \dots, \alpha_x, \frac{k-1}{2}, \frac{k-3}{2}, \dots, -\frac{k-1}{2}, \beta_1, \dots, \beta_y, \gamma_1, \dots, \gamma_z, \delta_1, \dots, \delta_w\bigg) + \bigg( \frac{m_0}{2}, \dots, \frac{m_0}{2} \bigg)
\]
with 
\begin{itemize}
\item $\alpha_i, \gamma_j > 0$ and $\beta_i, \delta_j < 0$;
\item $x+y+k = p$ and $z+w = q$;
\item $x+w = r$ and $z+y = s$,
\end{itemize}
or 
\[
 \lambda = \bigg(\alpha_1, \dots, \alpha_x, \beta_1, \dots, \beta_y, \gamma_1, \dots, \gamma_z, \frac{k-1}{2}, \frac{k-3}{2}, \dots, -\frac{k-1}{2}, \delta_1, \dots, \delta_w \bigg) + \bigg( \frac{m_0}{2}, \dots, \frac{m_0}{2} \bigg)
\]
with 
\begin{itemize}
\item $\alpha_i, \gamma_j > 0$ and $\beta_i, \delta_j < 0$;
\item $x+y = p$ and $z+w+k = q$;
\item $x+w = r$ and $z+y = s$.
\end{itemize}
Then $\lambda'$ is given by
\[
 \lambda' = (\alpha_1, \dots, \alpha_x, \delta_1, \dots, \delta_w, \gamma_1, \dots, \gamma_z, \beta_1, \dots, \beta_y) + \bigg( \frac{n_0}{2}, \dots, \frac{n_0}{2} \bigg).
\]
\end{enumerate}
\end{thm}

\begin{rem}
Theorem \ref{t:main} follows from a result of Li \cite{li90} under the assumption that $m \ge n$ and 
\[
 \alpha_x, - \beta_1, \gamma_z, -\delta_1 \ge \frac{m-n+1}{2}.
\]
\end{rem}

\subsection{Nonvanishing of theta lifts}
\label{ss:nonvanishing}

For the convenience of the reader, we include here a combinatorial criterion for the nonvanishing of $\theta_{r,s}(\pi)$ due to Atobe \cite{atobe}, which we will use in the proof of Theorem \ref{t:main}.
Note that the choice of $\psi$ in \cite[\S 3.3]{atobe} is not made explicit, but in fact, it agrees with our choice (see \cite[Lemma 1.4.5]{paul1} and Lemma \ref{l:K-type-corresp} below).

Fix $k_0 = -1$ or $0$.
We consider the theta lifting from $\U(W)$ to $\U(V)$, where $W$ is an $n$-dimensional skew-Hermitian space over $\C$ and $V$ varies over $m$-dimensional Hermitian spaces over $\C$ with 
\[
 m \equiv n + k_0 \bmod 2.
\]
Let $(p,q)$ be the signature of $W$, so that $p+q = n$.
Let $\pi$ be a discrete series representation of $\U(W) = \U(p,q)$ with Harish-Chandra parameter $\lambda$.
Following \cite[Definition 1.6]{atobe}, we will define some invariants of $\lambda$ which depend on $k_0$ and $\chi_V$.
\begin{enumerate}
\item
Write
\[
 \lambda = \lambda_0 + \bigg( \frac{m_0}{2}, \dots, \frac{m_0}{2} \bigg)
\]
with $\lambda_0 = (\lambda_{0,1}, \dots, \lambda_{0,n})$, so that 
\begin{itemize}
\item $\lambda_{0,i} \in \Z + \frac{k_0 - 1}{2}$;
\item $\lambda_{0,i} \ne \lambda_{0,j}$ if $i \ne j$;
\item $\lambda_{0,1} > \dots > \lambda_{0,p}$ and $\lambda_{0,p+1} > \dots > \lambda_{0,n}$.
\end{itemize}
\item
Let $k_\lambda$ be the largest positive integer with $k_\lambda \equiv k_0 \bmod 2$ such that 
\[
 \bigg\{ \frac{k_\lambda-1}{2}, \frac{k_\lambda-3}{2}, \dots, -\frac{k_\lambda-1}{2} \bigg\}
 \subset \{ \lambda_{0,1}, \dots, \lambda_{0,p} \}
\]
or 
\[
 \bigg\{ \frac{k_\lambda-1}{2}, \frac{k_\lambda-3}{2}, \dots, -\frac{k_\lambda-1}{2} \bigg\}
 \subset \{ \lambda_{0,p+1}, \dots, \lambda_{0,n} \}.
\]
If such an integer does not exist, we put $k_\lambda = k_0$.
\item
If $k_\lambda \ge 1$, we write
\[
 \lambda_0 = \bigg( \alpha_1, \dots, \alpha_x, \frac{k_\lambda-1}{2}, \frac{k_\lambda-3}{2}, \dots, -\frac{k_\lambda-1}{2}, \beta_1, \dots, \beta_y, \gamma_1, \dots, \gamma_z, \delta_1, \dots, \delta_w \bigg)
\]
with
\begin{itemize}
\item $\alpha_i, \gamma_j > 0$ and $\beta_i, \delta_j < 0$;
\item $x+y+k_\lambda = p$ and $z+w = q$,
\end{itemize}
or 
\[
 \lambda_0 = \bigg( \alpha_1, \dots, \alpha_x, \beta_1, \dots, \beta_y, \gamma_1, \dots, \gamma_z, \frac{k_\lambda-1}{2}, \frac{k_\lambda-3}{2}, \dots, -\frac{k_\lambda-1}{2}, \delta_1, \dots, \delta_w \bigg)
\]
with
\begin{itemize}
\item $\alpha_i, \gamma_j > 0$ and $\beta_i, \delta_j < 0$;
\item $x+y = p$ and $z+w+k_\lambda = q$.
\end{itemize}
If $k_\lambda \le 0$, we write
\[
 \lambda_0 = (\alpha_1, \dots, \alpha_x, \beta_1, \dots, \beta_y, \gamma_1, \dots, \gamma_z, \delta_1, \dots, \delta_w) 
\]
with 
\begin{itemize}
\item $\alpha_i, \gamma_j > 0$ and $\beta_i, \delta_j < 0$;
\item $x+y = p$ and $z+w = q$.
\end{itemize}
\item
Put
\[
 r_\lambda = x+w, \quad
 s_\lambda = z+y.
\]
\item 
Define a finite subset $\XX_\lambda$ of $\frac{1}{2} \Z \times \{ \pm 1 \}$ by
\[
 \XX_\lambda = \{ (\lambda_{0,1}, +1), \dots, (\lambda_{0,p}, +1) \} \cup \{ (\lambda_{0,p+1}, -1), \dots, (\lambda_{0,n}, -1) \}.
\]
\item
Define a sequence 
\[
 \XX_\lambda = \XX^{(0)}_\lambda \supset \XX^{(1)}_\lambda \supset \dots \supset \XX^{(j)}_\lambda \supset \cdots
\]
inductively as follows.
Write the image of $\XX^{(j)}_\lambda$ under the projection $\frac{1}{2} \Z \times \{ \pm 1 \} \rightarrow \frac{1}{2} \Z$ as 
\[
 \{ \xi_1, \xi_2, \dots \}
\]
with $\xi_1 > \xi_2 > \cdots$ and define a subset $\XX^{(j+1)}_\lambda$ of $\XX^{(j)}_\lambda$ by 
\[
 \XX^{(j+1)}_\lambda = \XX^{(j)}_\lambda \smallsetminus
 \bigg( \bigcup_i \{ (\xi_i, +1), (\xi_{i+1}, -1) \} \bigg), 
\]
where $i$ runs over indices satisfying one of the following conditions:
\begin{itemize}
\item $\xi_i \in \{ \alpha_1, \dots, \alpha_x \}$ and $\xi_{i+1} \in \{ \gamma_1, \dots, \gamma_z \}$;
\item $\xi_i \in \{ \beta_1, \dots, \beta_y \}$ and $\xi_{i+1} \in \{ \delta_1, \dots, \delta_w \}$.
\end{itemize}
\item
Put
\[
 \XX^{(\infty)}_\lambda = \XX^{(j)}_\lambda = \XX^{(j+1)}_\lambda = \cdots
\]
with some sufficiently large $j$.
\item 
For any integer $t$, we define subsets $\CC^\pm_\lambda(t)$ of $\XX^{(\infty)}_\lambda$ by
\begin{align*}
 \CC^+_\lambda(t) & = \left\{ (\xi,+1) \in \XX^{(\infty)}_\lambda \, \left| \, 0 \le \frac{k_\lambda-1}{2} + \xi < t \right. \right\}, \\ 
 \CC^-_\lambda(t) & = \left\{ (\xi,-1) \in \XX^{(\infty)}_\lambda \, \left| \, 0 \le \frac{k_\lambda-1}{2} - \xi < t \right. \right\}.
\end{align*}
\end{enumerate}

Then we have the following criterion for the nonvanishing of theta lifts.

\begin{thm}[{Atobe \cite[Theorem 1.7]{atobe}}]
\label{t:nonvanishing}
Let $\pi$ be a discrete series representation of $\U(p,q)$ with Harish-Chandra parameter $\lambda$.
Let $l,t$ be integers with $t \ge 1$.
\begin{enumerate}
\item
Assume that $k_\lambda = -1$.
Then
\begin{itemize}
\item $\theta_{r_\lambda+l+1, s_\lambda+l}(\pi)$ is nonzero if and only if $l \ge 0$;
\item $\theta_{r_\lambda+l+2t+1, s_\lambda+l}(\pi)$ is nonzero if and only if
\[
 l \ge 0, \quad
 \# \CC^+_\lambda(l+t) \le l, \quad
 \# \CC^-_\lambda(l+t) \le l.
\]
\end{itemize}
\item
Assume that $k_\lambda \ge 0$.
Then
\begin{itemize}
\item $\theta_{r_\lambda+l, s_\lambda+l}(\pi)$ is nonzero if and only if $l \ge 0$;
\item $\theta_{r_\lambda+l+2t, s_\lambda+l}(\pi)$ is nonzero if and only if
\[
 l \ge k_\lambda, \quad
 \# \CC^+_\lambda(l+t) \le l, \quad
 \# \CC^-_\lambda(l+t) \le l.
\]
\end{itemize}
\end{enumerate}
\end{thm}

\begin{rem}
\label{r:nonvanishing}
To determine the nonvanishing of $\theta_{r,s}(\pi)$, we apply Theorem \ref{t:nonvanishing} directly if $r-r_\lambda \ge s-s_\lambda$, but after replacing $(r,s)$ by $(s,r)$ and $\pi$ by $\bar{\pi} \otimes (\chi_V \circ \det)$ if $r-r_\lambda < s-s_\lambda$.
Indeed, we have $\theta_{r,s}(\pi) \ne 0$ if and only if $\theta_{s,r}(\bar{\pi} \otimes (\chi_V \circ \det)) \ne 0$ by Lemma \ref{l:complex-conjugate-theta}, while we have
\[
 k_{\lambda'} = k_\lambda, \quad
 (r_{\lambda'}, s_{\lambda'}) = (s_\lambda, r_\lambda)
\] 
for the Harish-Chandra parameter $\lambda'$ of $\bar{\pi} \otimes (\chi_V \circ \det)$.
\end{rem}

\begin{rem}
\label{r:nonvanishing-li}
Theorem \ref{t:nonvanishing} is consistent with a result of Li \cite{li90} on the nonvanishing of the theta lifts of discrete series representations with sufficiently regular infinitesimal character.
More precisely, he showed that the theta lift $\theta_{r,s}(\pi)$ to $\U(r,s)$ of a discrete series representation $\pi$ of $\U(p,q)$ is nonzero if
\[
 m \ge n
\]
with $n=p+q$ and $m=r+s$, and the Harish-Chandra parameter $\lambda$ of $\pi$ is of the form
\[
 \lambda = (\alpha_1, \dots, \alpha_x, \beta_1, \dots, \beta_y, \gamma_1, \dots, \gamma_z, \delta_1, \dots, \delta_w) + \bigg( \frac{m_0}{2}, \dots, \frac{m_0}{2} \bigg)
\]
with
\begin{itemize}
\item $x+y = p$ and $z+w = q$;
\item $x+w \le r$ and $z+y \le s$,
\end{itemize}
and
\begin{equation}
\label{eq:suff-reg}
 \alpha_x, - \beta_1, \gamma_z, -\delta_1 \ge \frac{m-n+1}{2}.
\end{equation}
If $m=n$ (and hence $k_\lambda$ is a nonnegative even integer), then we have $n = r_\lambda + s_\lambda + k_\lambda$ and 
\[
 (r,s) = (r_\lambda + \tfrac{k_\lambda}{2}, s_\lambda + \tfrac{k_\lambda}{2}).
\]
This is consistent with Theorem \ref{t:nonvanishing}.
Thus assume that $m > n$.
Then we have $k_\lambda = -1$ or $0$ by \eqref{eq:suff-reg}, so that $r_\lambda \le r$ and $s_\lambda \le s$.
As in Remark \ref{r:nonvanishing}, we may assume that $r-r_\lambda \ge s-s_\lambda$, in which case we have
\[
 (r,s) =  
 \begin{cases}
  (r_\lambda+l+2t+1, s_\lambda+l) & \text{if $k_\lambda = -1$;} \\
  (r_\lambda+l+2t, s_\lambda+l) & \text{if $k_\lambda = 0$}
 \end{cases}
\]
for some nonnegative integers $l,t$.
To check the consistency with Theorem \ref{t:nonvanishing}, it suffices to show that
\[
 \CC_\lambda^\pm(l+t) = \varnothing.
\]
Indeed, if $(\xi, \pm 1) \in \CC_\lambda^\pm(l+t)$, then we have
\[
\begin{cases}
 1 \le |\xi| \le l+t & \text{if $k_\lambda = -1$;} \\
 \frac{1}{2} \le |\xi| \le l+t-\frac{1}{2} & \text{if $k_\lambda = 0$.}
\end{cases} 
\]
But we have $n = r_\lambda + s_\lambda$ and 
\[
 m = 
\begin{cases}
 r_\lambda + s_\lambda + 2l + 2t + 1 & \text{if $k_\lambda = -1$;} \\
 r_\lambda + s_\lambda + 2l + 2t & \text{if $k_\lambda = 0$,}
\end{cases} 
\]
so that 
\[
 \frac{m-n+1}{2} = 
\begin{cases}
 l+t+1 & \text{if $k_\lambda = -1$;} \\
 l+t+\frac{1}{2} & \text{if $k_\lambda = 0$.}
\end{cases} 
\]
Hence by \eqref{eq:suff-reg}, we have $\CC_\lambda^\pm(l+t) = \varnothing$.
\end{rem}

\section{Local $L$- and $A$-packets}
\label{s:packets}

In this section, we describe the representations in some local $L$- and $A$-packets for unitary groups explicitly.

\subsection{Parameters and packets}

Let $F$ be a local field of characteristic zero and $W_F$ the Weil group of $F$.
Put
\[
 L_F =
 \begin{cases}
  W_F & \text{if $F$ is archimedean;} \\
  W_F \times \SL_2(\C) & \text{if $F$ is nonarchimedean.}
 \end{cases}
\]
Let $E$ be a quadratic extension of $F$.
As in \cite[\S 8]{ggp1}, we may regard an $L$-parameter $\phi : L_F \rightarrow {}^L \U_n$ (resp.~an $A$-parameter $\phi : L_F \times \SL_2(\C) \rightarrow {}^L \U_n$) for $\U_n$, where $\U_n$ stands for the unitary group of any $n$-dimensional Hermitian or skew-Hermitian space over $E$ and ${}^L \U_n =  \GL_n(\C) \rtimes W_F$ is the $L$-group of $\U_n$, as an $n$-dimensional conjugate-self-dual representation of $L_E$ (resp.~$L_E \times \SL_2(\C)$) with sign $(-1)^{n-1}$.
For any such a parameter $\phi$, we denote by $S_\phi$ the component group of the centralizer of the image of $\phi$ in $\GL_n(\C)$ and by $\widehat{S}_\phi$ the group of characters of $S_\phi$.
Note that $S_\phi$ is a finitely generated free $\Z/2\Z$-module.
We denote by $\mathbbm{1}$ the trivial character of $S_\phi$.
For any positive integer $d$, we denote by $S_d$ the unique $d$-dimensional irreducible representation of $\SL_2(\C)$.

Fix $\varepsilon = \pm 1$.
Let $V$ be an $n$-dimensional $\varepsilon$-Hermitian space over $E$.
Then the local Langlands correspondence \cite{mok,kmsw,mr18} gives a partition of the set $\Irr \U(V)$ of equivalence classes of irreducible representations of $\U(V)$ into finite sets called $L$-packets:
\[
 \Irr \U(V) = \bigsqcup_\phi \Pi_\phi(\U(V)), 
\]
where $\phi$ runs over $L$-parameters for $\U_n$.
Moreover, given the choice of a Whittaker datum, there exists a canonical bijection
\[
 \bigsqcup_V \Pi_\phi(\U(V)) \longleftrightarrow \widehat{S}_\phi,
\]
where $V$ runs over isometry classes of $n$-dimensional $\varepsilon$-Hermitian spaces over $E$.
We denote by $\pi(\phi, \eta)$ the irreducible representation associated to $\eta \in \widehat{S}_\phi$.

To any $A$-parameter $\phi'$ for $\U_n$, Arthur's endoscopic classification \cite{mok,kmsw} assigns a finite set called an $A$-packet
\[
 \Pi_{\phi'}(\U(V))
\]
consisting of (possibly zero, possibly reducible) semisimple representations of $\U(V)$ of finite length.
Given the choice of a Whittaker datum, the representations in $\Pi_{\phi'}(\U(V))$ are indexed by $\widehat{S}_{\phi'}$.
We denote by $\sigma(\phi', \eta')$ the representation associated to $\eta' \in \widehat{S}_{\phi'}$.

\subsection{Whittaker data}
\label{ss:whit}

To index the representations in $L$- and $A$-packets as in the previous subsection, we take the following Whittaker datum in this paper.
If $n$ is odd, then there is a unique Whittaker datum.
Thus assume that $n$ is even.
Then by \cite[Proposition 12.1]{ggp1}, the Whittaker data are parametrized by $\mathrm{N}_{E/F}(E^\times)$-orbits of nontrivial additive characters of $E/F$ (resp.~$F$) if $\varepsilon = +1$ (resp.~$\varepsilon = -1$).
On the other hand, we have fixed an element $\delta \in E^\times$ such that $\Tr_{E/F}(\delta) = 0$ and a nontrivial additive character $\psi$ of $F$.
Define a nontrivial additive character $\psi^E$ of $E/F$ by $\psi^E(x) = \psi(\frac{1}{2} \Tr_{E/F}(\delta x))$.
Following \cite[\S 2.4]{gi2}, we take the Whittaker datum associated to $\psi^E$ (resp.~$\psi$) if $\varepsilon = +1$ (resp.~$\varepsilon = -1$).

If $F = \R$, we always assume that $\delta = \sqrt{-1}$ and $\psi(x) = e^{- 2 \pi \sqrt{-1} x}$.
Then our Whittaker datum agrees with the Whittaker datum $\mathfrak{w}_+$ as in \cite[\S A.3]{atobe}.
Moreover, by \cite[Theorem A.4]{atobe}, it also agrees with the Whittaker datum as in \cite[Remarque 4.5]{mr19}.

\subsection{The real case}
\label{ss:packets-real}

Suppose that $F = \R$.
For any $\kappa \in \frac{1}{2} \Z$, we define a character $\chi_\kappa$ of $W_\C = \C^\times$ by 
\[
 \chi_\kappa(z) = \left( \frac{z}{\sqrt{z \bar{z}}} \right)^{2 \kappa}.
\]

\subsubsection{Some $L$-packets}
\label{sss:packets-real-ds}

We consider the Vogan $L$-packet $\bigsqcup_{p+q=n} \Pi_\phi(\U(p,q))$, where $\phi$ is an $L$-parameter for $\U_n$ of the form
\[
 \phi = \chi_{\kappa_1} \oplus \dots \oplus \chi_{\kappa_n}
\]
with 
\begin{itemize}
\item $\kappa_i \in \Z + \frac{n-1}{2}$;
\item $\kappa_1 > \dots > \kappa_n$.
\end{itemize}
Then $S_\phi$ is a free $\Z / 2 \Z$-module of the form
\[
 S_\phi = (\Z / 2 \Z) e_1 \oplus \dots \oplus (\Z / 2 \Z) e_n,
\]
where $e_i$ corresponds to $\chi_{\kappa_i}$.
Let $\eta \in \widehat{S}_\phi$.
Put
\begin{alignat*}{2}
 I^+ & = \{ i \, | \, \eta(e_i) = (-1)^{i-1} \}, \quad & p & = \# I^+, \\
 I^- & = \{ i \, | \, \eta(e_i) = (-1)^i \}, \quad & q & = \# I^-,
\end{alignat*}
and write
\begin{itemize}
\item $\{ \kappa_i \, | \, i \in I^+ \} = \{ \lambda_1, \dots, \lambda_p \}$ with $\lambda_1 > \dots > \lambda_p$;
\item $\{ \kappa_i \, | \, i \in I^- \} = \{ \lambda_{p+1}, \dots, \lambda_n \}$ with $\lambda_{p+1} > \dots > \lambda_n$.
\end{itemize}
Note that
\[
 \eta(e_1 + \dots + e_n) = (-1)^{\frac{1}{2}(p-q)(p-q-1)}.
\]
Then by \cite[Theorem A.4]{atobe}, $\pi(\phi, \eta)$ is the discrete series representation of $\U(p,q)$ with Harish-Chandra parameter
\[
 \lambda = (\lambda_1, \dots, \lambda_n).
\]

\subsubsection{Some $A$-packets}
\label{sss:packets-real-coh}

We consider the $A$-packet $\Pi_{\phi'}(\U(r,s))$ with $r+s=m$, where $\phi'$ is an $A$-parameter for $\U_m$ of the form
\[
 \phi' = \chi_{\mu_1} \oplus \cdots \oplus \chi_{\mu_n} \oplus (\chi_{\mu_0} \boxtimes S_{m-n})
\]
with $n < m$ and
\begin{itemize}
\item $\mu_i \in \Z + \frac{m-1}{2}$ for $i \ne 0$;
\item $\mu_0 \in \Z + \frac{n}{2}$;
\item $\mu_1 > \dots > \mu_{i_0-1} > \mu_0 \ge \mu_{i_0} > \dots > \mu_n$.
\end{itemize}
Then $S_{\phi'}$ is a quotient of a free $\Z/2\Z$-module $\widetilde{S}_{\phi'}$ of the form
\[
 \widetilde{S}_{\phi'} = (\Z / 2 \Z) e'_1 \oplus \dots \oplus (\Z / 2 \Z) e'_n \oplus (\Z / 2 \Z) e_0',
\]
where $e'_i$ corresponds to $\chi_{\mu_i}$ (resp.~$\chi_{\mu_0} \boxtimes S_{m-n}$) if $i \ne 0$ (resp.~$i=0$).
In fact, we have $S_{\phi'} = \widetilde{S}_{\phi'}$ unless $\mu_0 = \mu_{i_0}$ and $m-n=1$, in which case we may identify $\widehat{S}_{\phi'}$ with the group of characters $\eta'$ of $\widetilde{S}_{\phi'}$ satisfying
\[
 \eta'(e_0') = \eta'(e_{i_0}').
\]
Let $\eta' \in \widehat{S}_{\phi'}$.
Define $\rr = (r_1,\dots,r_{n+1}), \ss = (s_1,\dots,s_{n+1})$ by
\begin{align*}
 r_i & =  
 \begin{cases}
  1 & \text{if $i<i_0$ and $\eta'(e_i') = (-1)^{i-1}$;} \\
  0 & \text{if $i<i_0$ and $\eta'(e_i') = (-1)^i$;} \\
  1 & \text{if $i>i_0$ and $\eta'(e_{i-1}') = (-1)^{i+m-n-2}$;} \\
  0 & \text{if $i>i_0$ and $\eta'(e_{i-1}') = (-1)^{i+m-n-1}$,}
 \end{cases} \\
 s_i & =  
 \begin{cases}
  0 & \text{if $i<i_0$ and $\eta'(e_i') = (-1)^{i-1}$;} \\
  1 & \text{if $i<i_0$ and $\eta'(e_i') = (-1)^i$;} \\
  0 & \text{if $i>i_0$ and $\eta'(e_{i-1}') = (-1)^{i+m-n-2}$;} \\
  1 & \text{if $i>i_0$ and $\eta'(e_{i-1}') = (-1)^{i+m-n-1}$,}
 \end{cases} \\
 r_{i_0} & = r - r_1 - \dots - r_{i_0-1} - r_{i_0+1} \dots - r_{n+1}, \\
 s_{i_0} & = s - s_1 - \dots - s_{i_0-1} - s_{i_0+1} \dots - s_{n+1}.
\end{align*}
Note that $r_{i_0} + s_{i_0} = m-n$.
Then by \cite[Th\'eor\`eme 1.1]{mr19}, the representation $\sigma(\phi', \eta')$ of $\U(r,s)$ is nonzero only if
\[
 r_{i_0}, s_{i_0} \ge 0
\]
and 
\begin{equation}
\label{eq:eta'}
 \eta'(e_1' + \dots + e_n' + e_0') = (-1)^{\frac{1}{2}(r-s)(r-s-1)}
\end{equation}
(see also Lemma \ref{l:eta'} below), in which case we have
\[
 \sigma(\phi', \eta') = A_\q(\lambda).
\]
Here $\q = \q_{\rr,\ss}$ is the $\theta$-stable parabolic subalgebra of $\u(r,s)$ as in \S \ref{ss:cohom} and $\lambda$ is the $1$-dimensional representation of $\l_{\rr,\ss}$ in the weakly fair range given by
\[
 \lambda = (\underbrace{\lambda_1,\dots,\lambda_1}_{r_1}, \dots, \underbrace{\lambda_{n+1},\dots,\lambda_{n+1}}_{r_{n+1}}, \underbrace{\lambda_1,\dots,\lambda_1}_{s_1}, \dots, \underbrace{\lambda_{n+1},\dots,\lambda_{n+1}}_{s_{n+1}})
\]
with 
\[
 \lambda_i = 
 \begin{cases}
  \mu_i - \frac{m-1}{2} + i-1 & \text{if $i < i_0$;} \\
  \mu_0 - \frac{n}{2} + i_0-1 & \text{if $i = i_0$;} \\
  \mu_{i-1} - \frac{m-1}{2} + i+m-n-2 & \text{if $i > i_0$.}
 \end{cases}
\]
(Note that there is a typo in \cite[(4-2)]{mr19}: $(t_i+a_i-N)/2 - a_{<i}$ should be $(t_i+a_i-N)/2 + a_{<i}$.)
Moreover, if two representations $\sigma(\phi', \eta_1'), \sigma(\phi', \eta_2')$ with $\eta_1', \eta_2' \in \widehat{S}_{\phi'}$ are nonzero and isomorphic, then we have $\eta_1' = \eta_2'$.

\begin{lem}
\label{l:eta'}
Let $\eta' \in \widehat{S}_{\phi'}$ and define $\rr = (r_1,\dots,r_{n+1}), \ss = (s_1,\dots,s_{n+1})$ as above.
Then $\eta'$ satisfies \eqref{eq:eta'} if and only if
\[
 \eta'(e_0') = (-1)^{r_{i_0}(i_0-1) + s_{i_0} i_0 + \frac{1}{2}(m-n)(m-n-1)}.
\]
\end{lem}

\begin{proof}
It suffices to show that 
\[
 \eta'(e_1' + \dots + e_n') \cdot (-1)^{r_{i_0}(i_0-1) + s_{i_0} i_0 + \frac{1}{2}(m-n)(m-n-1)} \cdot (-1)^{\frac{1}{2}(r-s)(r-s-1)} = 1.
\]
We may write $\eta'(e_1' + \dots + e_n') = (-1)^j$, where
\begin{align*}
 j & = \sum_{i=1}^{i_0-1} (i-1+s_i) + \sum_{i=i_0+1}^{n+1} (i+m-n-2+s_i) \\
 & = \frac{1}{2}(i_0-1)(i_0-2) + \frac{1}{2}(n-i_0+1)(2m-n+i_0-2) + s-s_{i_0} \\
 & = \frac{1}{2}n(2m-n-1) - (m-n)(i_0-1) + s-s_{i_0} \\
 & = \frac{1}{2}n(2m-n-1) - r_{i_0}(i_0-1) - s_{i_0} i_0 + s.
\end{align*}
Then we have
\begin{align*}
 & j + r_{i_0}(i_0-1) + s_{i_0} i_0 + \frac{1}{2}(m-n)(m-n-1) + \frac{1}{2}(r-s)(r-s-1) \\
 & = \frac{1}{2}n(2m-n-1) + s + \frac{1}{2}(m-n)(m-n-1) + \frac{1}{2}(r-s)(r-s-1) \\
 & = \frac{1}{2} m (m-1) + s + \frac{1}{2}(r-s)(r-s-1) \\
 & = \frac{1}{2} (r+s) (r+s-1) + s + \frac{1}{2}(r-s)(r-s-1) \\
 & = r(r-1) + s(s+1) \\
 & \equiv 0 \bmod 2.
\end{align*}
This implies the assertion.
\end{proof}

\subsection{The nonarchimedean case}

Suppose that $F$ is nonarchimedean.
Recall that given a positive integer $n$, there are precisely two $n$-dimensional $\varepsilon$-Hermitian spaces $V_n^+$ and $V_n^-$ over $E$ (up to isometry).
Consider a parabolically induced representation
\[
 \Ind^{\U(V_n^\epsilon)}_P(\tau_1 \boxtimes \dots \boxtimes \tau_k \boxtimes \pi_0),
\]
where 
\begin{itemize}
\item $P$ is a parabolic subgroup of $\U(V_n^\epsilon)$ with Levi component $\GL_{n_1}(E) \times \dots \times \GL_{n_k}(E) \times \U(V_{n_0}^\epsilon)$;
\item $\tau_i$ is an irreducible essentially tempered representation of $\GL_{n_i}(E)$;
\item $\pi_0$ is an irreducible tempered representation of $\U(V_{n_0}^\epsilon)$.
\end{itemize}
If the representation above is a standard module, we denote its unique irreducible quotient by
\[
 J(\tau_1, \dots, \tau_k, \pi_0).
\]

\subsubsection{Some $L$-packets}

We consider the Vogan $L$-packet $\Pi_\phi(\U(V_n^+)) \sqcup \Pi_\phi(\U(V_n^-))$, where $\phi$ is an $L$-parameter for $\U_n$ of the form
\[
 \phi = \chi_1 \oplus \dots \oplus \chi_n
\]
with (not necessarily distinct) conjugate-self-dual characters $\chi_i$ of $E^\times$ with sign $(-1)^{n-1}$.
Then $\pi(\phi,\mathbbm{1})$ is an irreducible tempered representation of $\U(V_n^+)$.
For more properties, we refer the reader to \cite[\S 2.5]{gi2}.

\subsubsection{Some $A$-packets}

We consider the $A$-packet $\Pi_{\phi'}(\U(V_m^+))$, where $\phi'$ is an $A$-parameter for $\U_m$ of the form
\[
 \phi' = \chi_1 \oplus \dots \oplus \chi_n \oplus (\chi_0 \boxtimes S_{m-n})
\]
with $n < m$ and (not necessarily distinct) conjugate-self-dual characters $\chi_i$ of $E^\times$ with sign
\[
\begin{cases}
 (-1)^{m-1} & \text{if $i \ne 0$;} \\
 (-1)^n & \text{if $i=0$.}
\end{cases} 
\]
Then by \cite[\S 4.1]{mr18}, the representation $\sigma(\phi', \eta')$ of $\U(V_m^+)$ with $\eta' \in \widehat{S}_{\phi'}$ is either zero or irreducible.
Moreover, if two representations $\sigma(\phi', \eta_1'), \sigma(\phi', \eta_2')$ with $\eta_1', \eta_2' \in \widehat{S}_{\phi'}$ are nonzero and isomorphic, then we have $\eta_1' = \eta_2'$.

\begin{lem}
\label{l:local-Apacket}
\begin{enumerate}
\item 
If $m \equiv n \bmod 2$, then we have
\[
 \sigma(\phi', \mathbbm{1}) = J(\chi_0 | \cdot |^{\frac{1}{2}(m-n-1)}, \chi_0 | \cdot |^{\frac{1}{2}(m-n-3)}, \dots, \chi_0 | \cdot |^{\frac{1}{2}}, \pi(\phi_0, \mathbbm{1}))
\]
with an $L$-parameter $\phi_0 = \chi_1 \oplus \dots \oplus \chi_n$ for $\U_n$.
\item
If $m \not \equiv n \bmod 2$, then we have
\[
 \sigma(\phi', \mathbbm{1}) = J(\chi_0 | \cdot |^{\frac{1}{2}(m-n-1)}, \chi_0 | \cdot |^{\frac{1}{2}(m-n-3)}, \dots, \chi_0 | \cdot |^1, \pi(\phi_1, \mathbbm{1}))
\]
with an $L$-parameter $\phi_1 = \chi_1 \oplus \dots \oplus \chi_n \oplus \chi_0$ for $\U_{n+1}$.
\end{enumerate}
\end{lem}

\begin{proof}
The assertion follows from \cite[Proposition 8.4.1]{mok} and the irreducibility of $\sigma(\phi', \mathbbm{1})$.
\end{proof}

\subsection{The split case}

We also need to consider the case when $F$ is nonarchimedean and $E = F \times F$.
Recall that given a positive integer $n$, there is a unique $n$-dimensional $\varepsilon$-Hermitian space $V_n^+ = \VV_n \otimes_F E$ over $E$ (up to isometry), where $\VV_n$ is an $n$-dimensional vector space over $F$.
Via the isomorphism $\U(V_n^+) \cong \GL(\VV_n)$ induced by the first projection, we may regard an $L$-parameter $\phi : L_F \rightarrow {}^L \U_n$ (resp.~an $A$-parameter $\phi : L_F \times \SL_2(\C) \rightarrow {}^L \U_n$) for $\U_n$ as an $n$-dimensional representation of $L_F$ (resp.~$L_F \times \SL_2(\C)$).
For any such a parameter $\phi$, the component group $S_\phi$ is always trivial.

Let $\phi$ and $\phi'$ be $L$- and $A$-parameters for $\U_n$ and $\U_m$, respectively, of the form
\begin{align*}
 \phi & = \chi_1 \oplus \dots \oplus \chi_n, \\
 \phi' & = \chi_1 \oplus \dots \oplus \chi_n \oplus (\chi_0 \boxtimes S_{m-n})
\end{align*}
with $n < m$ and (not necessarily distinct) unitary characters $\chi_i$ of $F^\times$.
We denote by $\pi(\phi, \mathbbm{1})$ and $\sigma(\phi', \mathbbm{1})$ the unique representations of $\U(V_n^+) \cong \GL(\VV_n)$ and $\U(V_m^+) \cong \GL(\VV_m)$ in the $L$- and $A$-packets $\Pi_\phi(\U(V_n^+))$ and $\Pi_{\phi'}(\U(V_m^+))$, respectively.
Then we have
\begin{align*}
 \pi(\phi, \mathbbm{1}) & = \Ind^{\GL(\VV_n)}_{\mathcal{B}}(\chi_1 \boxtimes \dots \boxtimes \chi_n), \\
 \sigma(\phi', \mathbbm{1}) & = \Ind^{\GL(\VV_m)}_{\mathcal{P}}(\chi_1 \boxtimes \dots \boxtimes \chi_n \boxtimes (\chi_0 \circ \det)), 
\end{align*}
where $\mathcal{B}$ is a Borel subgroup of $\GL(\VV_n)$ and $\mathcal{P}$ is a parabolic subgroup of $\GL(\VV_m)$ with Levi component $(F^\times)^n \times \GL_{m-n}(F)$.
Note that the parabolically induced representations on the right-hand side are irreducible by \cite[Theorem 4.2]{zel}.

\section{Proof of the main theorem}
\label{s:proof}

In this section, we prove Theorem \ref{t:main}.

\subsection{Reduction to Theorem \ref{t:main}(\ref{item:main1})}

We first reduce Theorem \ref{t:main}\eqref{item:main2} to Theorem \ref{t:main}\eqref{item:main1}.
Thus we consider the theta lifting from $\U(p,q)$ to $\U(r,s)$ with $p+q=n$ and $r+s=m$ in the case
\[
 m \le n.
\]
If $m=n$, then Theorem \ref{t:main}\eqref{item:main2} follows from a result of Li \cite{li90} (see also \cite{paul1}).
Hence we may assume that $m < n$.

Let $\pi$ be a discrete series representation of $\U(p,q)$ with Harish-Chandra parameter $\lambda$.
We assume that its theta lift $\theta_{r,s}(\pi)$ to $\U(r,s)$ relative to $(\chi_V,\chi_W,\psi)$ is nonzero, where we take the data $(\chi_V,\chi_W,\psi)$ given in \S \ref{ss:main-theorem-setup}.
Then by Theorem \ref{t:nonvanishing}, we have $k_\lambda \ge 0$ and $(r,s) = (r_\lambda + l, s_\lambda + l)$ for some nonnegative integer $l$, where $k_\lambda, r_\lambda, s_\lambda$ are as defined in \S \ref{ss:nonvanishing}.
Put $k = n - m$.
Since $n = r_\lambda + s_\lambda + k_\lambda$ and $m = r_\lambda + s_\lambda + 2l$, we have 
\[
 k \le k_\lambda.
\]
Hence we may write either
\[
 \lambda = \bigg(\alpha_1, \dots, \alpha_x, \frac{k-1}{2}, \frac{k-3}{2}, \dots, -\frac{k-1}{2}, \beta_1, \dots, \beta_y, \gamma_1, \dots, \gamma_z, \delta_1, \dots, \delta_w\bigg) + \bigg( \frac{m_0}{2}, \dots, \frac{m_0}{2} \bigg)
\]
or 
\[
 \lambda = \bigg(\alpha_1, \dots, \alpha_x, \beta_1, \dots, \beta_y, \gamma_1, \dots, \gamma_z, \frac{k-1}{2}, \frac{k-3}{2}, \dots, -\frac{k-1}{2}, \delta_1, \dots, \delta_w \bigg) + \bigg( \frac{m_0}{2}, \dots, \frac{m_0}{2} \bigg)
\]
as in Theorem \ref{t:main}\eqref{item:main2}.

We only consider the first case; the second case is similar.
Let $\sigma$ be the discrete series representation of $\U(r,s)$ with Harish-Chandra parameter
\[
 \lambda' = (\alpha_1, \dots, \alpha_x, \delta_1, \dots, \delta_w, \gamma_1, \dots, \gamma_z, \beta_1, \dots, \beta_y) + \bigg( \frac{n_0}{2}, \dots, \frac{n_0}{2} \bigg).
\]
Since
\[
 \alpha_x, -\beta_1, \gamma_z, -\delta_1 \ge \frac{k+1}{2},  
\]
the theta lift $\theta_{p,q}(\sigma)$ to $\U(p,q)$ is nonzero by Remark \ref{r:nonvanishing-li}.
Hence we have 
\[
 \theta_{p,q}(\sigma) = A_{\q'}(\lambda'')
\]
if we admit Theorem \ref{t:main}\eqref{item:main1}, where $\q' = \q(\xx')$ is associated to
\[
 \xx' = (\alpha_1, \dots, \alpha_x, \underbrace{0, \dots, 0}_k, \beta_1, \dots, \beta_y, \gamma_1, \dots, \gamma_z, \delta_1, \dots, \delta_w)
\]
and $\lambda''$ is given by
\[
 \lambda'' = (\alpha_1'', \dots, \alpha_x'', \epsilon'', \dots, \epsilon'', \beta_1'', \dots, \beta_y'', \gamma_1'', \dots, \gamma_z'', \delta_1'', \dots, \delta_w'')
 + \bigg(\frac{m_0}{2}, \dots, \frac{m_0}{2}\bigg)
\]
with 
\begin{align*}
 \alpha_i'' & = \alpha_i - \frac{n+1}{2} + i + \# \{ k \, | \, \gamma_k > \alpha_i \}, \\ 
 \beta_j'' & = \beta_j + \frac{n-1}{2} + j-y - \# \{ l \, | \, \delta_l < \beta_j \}, \\
 \gamma_k'' & = \gamma_k -\frac{n+1}{2} + k + \# \{ i \, | \, \alpha_i > \gamma_k \}, \\
 \delta_l'' & = \delta_l + \frac{n-1}{2} + l-w - \# \{ j \, | \, \beta_j < \delta_l \}, \\
 \epsilon'' & = x+z - \frac{m}{2}.
\end{align*}
Since $\l(\xx')$ is contained in $\k$ (where $\k$ is the complexified Lie algebra of $K$) and $\lambda''$ is in the good range, $A_{\q'}(\lambda'')$ is the discrete series representation of $\U(p,q)$ with Harish-Chandra parameter $\lambda'' + \rho(\Psi)$, where $\Psi$ is the positive system of $\Delta$ determined by
\begin{itemize}
\item $\Delta_c^+ \subset \Psi$;
\item $\langle \alpha, \lambda'' \rangle > 0$ for all $\alpha \in \Psi$
\end{itemize}
and $\rho(\Psi)$ is half the sum of the roots in $\Psi$.
Moreover, we have $\lambda'' + \rho(\Psi) = \lambda$, so that $A_{\q'}(\lambda'') = \pi$.
This shows that $\theta_{p,q}(\sigma) = \pi$ and hence $\theta_{r,s}(\pi) = \sigma$, which reduces Theorem \ref{t:main}\eqref{item:main2} to Theorem \ref{t:main}\eqref{item:main1}.

The rest of this section is devoted to the proof of Theorem \ref{t:main}\eqref{item:main1}.

\subsection{Local theta lifting}

Let $F$ be a local field of characteristic zero and $E$ an \'etale quadratic algebra over $F$.
We consider the theta lifting from $\U(W)$ to $\U(V)$, where $W$ is an $n$-dimensional skew-Hermitian space over $E$ and $V$ is an $m$-dimensional Hermitian space over $E$ with $m>n$.

\subsubsection{The real case}
\label{sss:local-theta-real}

Suppose that $F=\R$ and $E=\C$.
Let $(p,q)$ and $(r,s)$ be the signatures of $W$ and $V$, respectively.
We take the data $(\chi_V, \chi_W, \psi)$ given in \S \ref{ss:main-theorem-setup}.

Let $\pi$ be a discrete series representation of $\U(p,q)$ with Harish-Chandra parameter $\lambda$.
Write 
\[
 \lambda = (\alpha_1, \dots, \alpha_x, \beta_1, \dots, \beta_y, \gamma_1, \dots, \gamma_z, \delta_1, \dots, \delta_w) + \bigg( \frac{m_0}{2}, \dots, \frac{m_0}{2} \bigg)
\]
with 
\begin{itemize}
\item $\alpha_i, \gamma_j > 0$ and $\beta_i, \delta_j \le 0$;
\item $x+y = p$ and $z+w = q$.
\end{itemize}

\begin{lem}
\label{l:local-theta-xyzw}
If the theta lift $\theta_{V,W,\chi_V,\chi_W,\psi}(\pi)$ to $\U(r,s)$ is nonzero, then we have
\[
 x+w \le r, \quad
 z+y \le s.
\]
\end{lem}

\begin{proof}
Let $k_\lambda, r_\lambda, s_\lambda$ be the integers as defined in \S \ref{ss:nonvanishing}.
As in Remark \ref{r:nonvanishing}, we may assume that $r - r_\lambda \ge s - s_\lambda$.
Since $m > n$ and 
\[
 (r_\lambda, s_\lambda) = 
\begin{cases}
 (x+w, z+y) & \text{if $k_\lambda = -1$}; \\
 (x+w-\frac{k_\lambda}{2}, z+y-\frac{k_\lambda}{2}) & \text{if $k_\lambda \ge 0$ and $k_\lambda$ is even}; \\
 (x+w-\frac{k_\lambda \pm 1}{2}, z+y-\frac{k_\lambda \mp 1}{2}) & \text{if $k_\lambda \ge 0$ and $k_\lambda$ is odd,}
\end{cases}
\]
the assertion follows from Theorem \ref{t:nonvanishing}.
\end{proof}

Define $L$- and $A$-parameters $\phi$ and $\phi'$ for $\U_n$ and $\U_m$, respectively, by 
\begin{align*}
 \phi & = \chi_{\kappa_1} \oplus \dots \oplus \chi_{\kappa_n}, \\
 \phi' & = \chi_{\kappa_1} \chi_V^{-1} \chi_W \oplus \dots \oplus \chi_{\kappa_n} \chi_V^{-1} \chi_W \oplus (\chi_W \boxtimes S_{m-n}), 
\end{align*}
where 
\begin{itemize}
\item $\kappa_1 > \dots > \kappa_{i_0-1} > \frac{m_0}{2} \ge \kappa_{i_0} > \dots > \kappa_n$;
\item $\{ \kappa_1 - \frac{m_0}{2}, \dots, \kappa_{i_0-1} - \frac{m_0}{2} \} = \{ \alpha_1, \dots, \alpha_x, \gamma_1, \dots, \gamma_z \}$;
\item $\{ \kappa_{i_0} - \frac{m_0}{2}, \dots, \kappa_n - \frac{m_0}{2} \} = \{ \beta_1, \dots, \beta_y, \delta_1, \dots, \delta_w \}$;
\item $i_0 = x+z+1$.
\end{itemize}
Then $\pi$ belongs to the $L$-packet $\Pi_\phi(\U(p,q))$.
As in \S \ref{ss:packets-real}, we write $S_\phi$ as 
\[
 S_\phi = (\Z/2\Z) e_1 \oplus \dots \oplus (\Z/2 \Z) e_n
\]
and $S_{\phi'}$ as a quotient of 
\[
 \widetilde{S}_{\phi'} = (\Z/2\Z) e'_1 \oplus \dots \oplus (\Z/2 \Z) e'_n \oplus (\Z/2 \Z) e'_0.
\]

\begin{lem}
\label{l:local-theta-1}
Assume that $x+w \le r$, $z+y \le s$, and
\[
 \alpha_x, - \beta_1, \gamma_z, -\delta_1 \ge \frac{m-n+1}{2}.
\]
Let $\eta$ be the character of $S_\phi$ associated to $\pi$ as in \S \ref{sss:packets-real-ds}, so that $\pi = \pi(\phi, \eta)$.
Then we have 
\[
 \theta_{V,W,\chi_V,\chi_W,\psi}(\pi(\phi, \eta)) = \sigma(\phi', \eta'), 
\]
where $\eta'$ is the character of $S_{\phi'}$ given by 
\[
 \eta'(e_i') = \zeta_i \times 
 \begin{cases}
  \eta(e_i) & \text{if $i \ne 0$;} \\
  \epsilon(V) \cdot \epsilon(W) & \text{if $i = 0$}
 \end{cases}
\]
with
\[
 \zeta_i = 
 \begin{cases}
  +1 & \text{if $m \equiv n \bmod 2$ and $0 < i < i_0$}; \\
  -1 & \text{if $m \equiv n \bmod 2$ and $i \ge i_0$;} \\
  +1 & \text{if $m \not \equiv n \bmod 2$ and $i \ne 0$}
 \end{cases}
\]
and 
\[
 \zeta_0 = \zeta_1 \cdots \zeta_n.
\]
(Note that $S_{\phi'} = \widetilde{S}_{\phi'}$ by assumption.)
\end{lem}

\begin{proof}
It follows from a result of Li \cite{li90} that
\[
 \theta_{V,W,\chi_V,\chi_W,\psi}(\pi(\phi, \eta)) = A_\q(\lambda'),
\]
where $\q$ and $\lambda'$ are as given in Theorem \ref{t:main}\eqref{item:main1}.
Hence it suffices to show that
\[
 \sigma(\phi', \eta') = A_\q(\lambda').
\]
Define $\rr = (r_1,\dots,r_{n+1}), \ss = (s_1,\dots,s_{n+1})$ as in \S \ref{sss:packets-real-coh}.
Then we have
\begin{align*}
 r_i & =  
 \begin{cases}
  1 & \text{if $i<i_0$ and $\eta(e_i) = (-1)^{i-1}$;} \\
  0 & \text{if $i<i_0$ and $\eta(e_i) = (-1)^i$;} \\
  1 & \text{if $i>i_0$ and $\eta(e_{i-1}) = (-1)^{i-1}$;} \\
  0 & \text{if $i>i_0$ and $\eta(e_{i-1}) = (-1)^{i-2}$,}
 \end{cases} \\
 s_i & =  
 \begin{cases}
  0 & \text{if $i<i_0$ and $\eta(e_i) = (-1)^{i-1}$;} \\
  1 & \text{if $i<i_0$ and $\eta(e_i) = (-1)^i$;} \\
  0 & \text{if $i>i_0$ and $\eta(e_{i-1}) = (-1)^{i-1}$;} \\
  1 & \text{if $i>i_0$ and $\eta(e_{i-1}) = (-1)^{i-2}$,}
 \end{cases}
\end{align*}
so that 
\[
 r_{i_0} = r - x - w, \quad
 s_{i_0} = s - z - y.
\]
Since $r_{i_0}, s_{i_0} \ge 0$ and 
\[
 \eta'(e_1' + \dots + e_n' + e_0') = \eta(e_1 + \dots + e_n) \cdot \epsilon(V) \cdot \epsilon(W) = \epsilon(V), 
\]
we have
\[
 \sigma(\phi', \eta') = A_{\tilde{\q}}(\tilde{\lambda}'), 
\]
where $\tilde{\q} = \q_{\rr,\ss}$ is the $\theta$-stable parabolic subalgebra of $\u(r,s)$ as in \S \ref{ss:cohom} and $\tilde{\lambda}'$ is the $1$-dimensional representation of $\l_{\rr,\ss}$ given by
\[
 \tilde{\lambda}' = (\underbrace{\tilde{\lambda}'_1,\dots,\tilde{\lambda}'_1}_{r_1}, \dots, \underbrace{\tilde{\lambda}'_{n+1},\dots,\tilde{\lambda}'_{n+1}}_{r_{n+1}}, \underbrace{\tilde{\lambda}'_1,\dots,\tilde{\lambda}'_1}_{s_1}, \dots, \underbrace{\tilde{\lambda}'_{n+1},\dots,\tilde{\lambda}'_{n+1}}_{s_{n+1}})
\]
with 
\[
 \tilde{\lambda}'_i = 
 \begin{cases}
  \kappa_i - \frac{m_0}{2} + \frac{n_0}{2} - \frac{m-1}{2} + i-1 & \text{if $i < i_0$;} \\
  \frac{n_0}{2} - \frac{n}{2} + i_0-1 & \text{if $i = i_0$;} \\
  \kappa_{i-1} - \frac{m_0}{2} + \frac{n_0}{2} - \frac{m-1}{2} + i+m-n-2 & \text{if $i > i_0$.}
 \end{cases}
\]
On the other hand, $\eta$ is given by 
\[
\begin{aligned}
 \eta(e_i) & = (-1)^{i-1} & & \Longleftrightarrow & \kappa_i - \tfrac{m_0}{2} & \in \{ \alpha_1, \dots, \alpha_x, \beta_1, \dots, \beta_y \}, \\
 \eta(e_i) & = (-1)^i  & & \Longleftrightarrow & \kappa_i - \tfrac{m_0}{2} & \in \{ \gamma_1, \dots, \gamma_z, \delta_1, \dots, \delta_w \}.
\end{aligned} 
\]
From this, we can deduce that $\tilde{\q} = \q$ and $\tilde{\lambda}' = \lambda'$.
This completes the proof.
\end{proof}

\subsubsection{The nonarchimedean case}

Suppose that $F$ is nonarchimedean and $E \ne F \times F$.
Let $\phi$ and $\phi'$ be $L$- and $A$-parameters for $\U_n$ and $\U_m$, respectively, of the form
\begin{align*}
 \phi & = \chi_1 \oplus \dots \oplus \chi_n, \\
 \phi' & = \chi_1 \chi_V^{-1} \chi_W \oplus \dots \oplus \chi_n \chi_V^{-1} \chi_W \oplus (\chi_W \boxtimes S_{m-n})
\end{align*}
with (not necessarily distinct) conjugate-self-dual characters $\chi_i$ of $E^\times$ with sign $(-1)^{n-1}$.
If $m \equiv n \bmod 2$, we assume further the condition on the $\epsilon$-factor
\begin{equation}
\label{eq:epsilon-assumption}
 \epsilon(\tfrac{1}{2}, \chi_i \chi_V^{-1}, \psi^E_2) = 1 
\end{equation}
for all $i$, where $\psi^E_2$ is the character of $E$ given by $\psi^E_2(x) = \psi(\Tr_{E/F}(\delta x))$.

\begin{lem}
\label{l:local-theta-2}
Assume that $\epsilon(V) = \epsilon(W) = +1$.
Then we have
\[
 \theta_{V,W,\chi_V,\chi_W,\psi}(\pi(\phi, \mathbbm{1})) = \sigma(\phi', \mathbbm{1}).
\]
\end{lem}

\begin{proof}
Write $\pi = \pi(\phi, \mathbbm{1})$ for brevity.
For any $\epsilon = \pm 1$, we define the first occurrence index $m^\epsilon(\pi)$ as the smallest nonnegative integer $m_0$ with $m_0 \equiv m \bmod 2$ such that $\theta_{V_{m_0}^\epsilon, W, \chi_V, \chi_W, \psi}(\pi) \ne 0$.
Put
\[
 m^\up(\pi) = \max \{ m^+(\pi), m^-(\pi) \}, \quad 
 m^\down(\pi) = \min \{ m^+(\pi), m^-(\pi) \}.
\]

Assume first that $m \equiv n \bmod 2$.
Then by \cite[Theorem 4.1]{atobe-gan} and \eqref{eq:epsilon-assumption}, we have
\[
 m^\up(\pi) = n+2, \quad
 m^\down(\pi) = n
\]
with $m^\down(\pi) = m^+(\pi)$.
Moreover, it follows from \cite[Theorem 4.3]{atobe-gan} that 
\[
 \theta_{V,W,\chi_V,\chi_W,\psi}(\pi) = J(\chi_W|\cdot|^{\frac{1}{2}(m-n-1)}, \chi_W|\cdot|^{\frac{1}{2}(m-n-3)}, \dots, \chi_W|\cdot|^{\frac{1}{2}}, \pi(\phi_0, \mathbbm{1}))
\]
with an $L$-parameter
\[
 \phi_0 = \chi_1 \chi_V^{-1} \chi_W \oplus \dots \oplus \chi_n \chi_V^{-1} \chi_W
\]
for $\U_n$.
Hence the assertion follows from Lemma \ref{l:local-Apacket}.

Assume next that $m \not \equiv n \bmod 2$.
Put
\[
 I = \{ i \, | \, \chi_i \ne \chi_V \}.
\]
Then by \cite[Theorem 4.1]{atobe-gan}, we have
\[
\begin{cases}
 m^\up(\pi) = n+1, \quad m^\down(\pi) = n+1 & \text{if $\# I$ is even;} \\
 m^\up(\pi) = n+3, \quad m^\down(\pi) = n-1 & \text{if $\# I$ is odd}
\end{cases} 
\]
with $m^\down(\pi) = m^+(\pi)$.
Moreover, it follows from \cite[Theorem 4.3]{atobe-gan} that
\[
 \theta_{V,W,\chi_V,\chi_W,\psi}(\pi) = J(\chi_W|\cdot|^{\frac{1}{2}(m-n-1)}, \chi_W|\cdot|^{\frac{1}{2}(m-n-3)}, \dots, \chi_W|\cdot|^1, \pi(\phi_1, \mathbbm{1}))
\]
with an $L$-parameter
\[
 \phi_1 =  \chi_1 \chi_V^{-1} \chi_W \oplus \dots \oplus \chi_n \chi_V^{-1} \chi_W \oplus \chi_W 
\]
for $\U_{n+1}$.
Hence the assertion follows from Lemma \ref{l:local-Apacket}.
\end{proof}

\subsubsection{The split case}

Suppose that $F$ is nonarchimedean and $E = F \times F$.
In this case, we may identify $\chi_V, \chi_W$ with unitary characters of $F^\times$ via the first projection.
Let $\phi$ and $\phi'$ be $L$- and $A$-parameters for $\U_n$ and $\U_m$, respectively, of the form
\begin{align*}
 \phi & = \chi_1 \oplus \dots \oplus \chi_n, \\
 \phi' & = \chi_1 \chi_V^{-1} \chi_W \oplus \dots \oplus \chi_n \chi_V^{-1} \chi_W \oplus (\chi_W \boxtimes S_{m-n})
\end{align*}
with (not necessarily distinct) unitary characters $\chi_i$ of $F^\times$.

\begin{lem}
\label{l:local-theta-3}
We have
\[
 \theta_{V,W,\chi_V,\chi_W,\psi}(\pi(\phi, \mathbbm{1})) = \sigma(\phi', \mathbbm{1}).
\]
\end{lem}

\begin{proof}
The assertion was proved by M\'inguez \cite{minguez}.
\end{proof}

\subsection{Global theta lifting}
\label{ss:global-theta}

Let $\F$ be a totally real number field with ad\`ele ring $\A = \A_\F$.
Let $\E$ be a totally imaginary quadratic extension of $\F$ and $\omega_{\E/\F}$ the quadratic character of $\A^\times/\F^\times$ associated to $\E/\F$ by global class field theory.
We consider the theta lifting from $\U(\W)$ to $\U(\V)$, where $\W$ is an $n$-dimensional skew-Hermitian space over $\E$ and $\V$ is an $m$-dimensional Hermitian space over $\E$ with $m>n$.
For simplicity, we assume that $\W$ and $\V$ are anisotropic.

Let $\omega_{\V,\W,\chi_\V,\chi_\W,\varPsi}$ be the Weil representation of $\U(\W)(\A) \times \U(\V)(\A)$ relative to $(\chi_\V,\chi_\W,\varPsi)$, where $\chi_\V, \chi_\W$ are characters of $\A_\E^\times / \E^\times$ such that $\chi_\V|_{\A^\times} = \omega_{\E/\F}^m, \chi_\W|_{\A^\times} = \omega_{\E/\F}^n$ and $\varPsi$ is a nontrivial additive character of $\A/\F$.
This is equipped with a natural equivariant map $\varphi \mapsto \theta(\varphi)$ to the space of left $\U(\W)(\F) \times \U(\V)(\F)$-invariant smooth functions on $\U(\W)(\A) \times \U(\V)(\A)$ of moderate growth.
For any irreducible automorphic representation $\varPi$ of $\U(\W)(\A)$, we denote by $\theta_{\V,\W,\chi_\V,\chi_\W,\varPsi}(\varPi)$ the space spanned by automorphic forms on $\U(\V)(\A)$ of the form
\[
 \theta(\varphi, f)(h) = \int_{\U(\W)(\F) \backslash \U(\W)(\A)} \theta(\varphi)(g,h) \overline{f(g)} \, dg
\]
for $\varphi \in \omega_{\V,\W,\chi_\V,\chi_\W,\varPsi}$ and $f \in \varPi$.
If $\theta_{\V,\W,\chi_\V,\chi_\W,\varPsi}(\varPi)$ is nonzero, then it follows from the Howe duality that $\theta_{\V,\W,\chi_\V,\chi_\W,\varPsi}(\varPi)$ is irreducible and isomorphic to $\bigotimes_v \theta_{\V_v,\W_v,\chi_{\V,v},\chi_{\W,v},\varPsi_v}(\varPi_v)$ (see \cite[Corollary 7.1.3]{kudla-rallis}).

We now discuss the nonvanishing of $\theta_{\V,\W,\chi_\V,\chi_\W,\varPsi}(\varPi)$.
For simplicity, we assume that $\varPi_v$ is tempered for all $v$ and that the partial standard $L$-function $L^S(s, \varPi, \chi_\V^{-1})$ of $\varPi$ twisted by $\chi_\V^{-1}$ is holomorphic and nonzero at $s = \frac{1}{2}(m-n+1)$, where $S$ is a sufficiently large finite set of places of $\F$.
Then the Rallis inner product formula, which is a consequence of the Siegel--Weil formula in the convergent range \cite{weil2, ichino}, says that
\[
 \langle \theta(\varphi_1, f_1), \theta(\varphi_2, f_2) \rangle = \frac{L^S(\frac{1}{2}(m-n+1), \varPi, \chi_\V^{-1})}{d^S(\frac{1}{2}(m-n))} \cdot \prod_{v \in S} \ZZ(\varphi_{1,v}, \varphi_{2,v}, f_{1,v}, f_{2,v})
\]
for $\varphi_i = \bigotimes_v \varphi_{i,v} \in \omega_{\V,\W,\chi_\V,\chi_\W,\varPsi}$ and $f_i = \bigotimes_v f_{i,v} \in \varPi$, where
\begin{itemize}
\item $\langle \cdot, \cdot \rangle$ is the Petersson inner product; 
\item $d^S(s) = \prod_{i=1}^n L^S(2s+i, \omega_{\E/\F}^{m-n+i})$, which is holomorphic and nonzero at $s = \frac{1}{2}(m-n)$;
\item $\ZZ(\varphi_{1,v}, \varphi_{2,v}, f_{1,v}, f_{2,v})$ is a certain local integral defined in \S \ref{ss:matcoeff-integral} below.
\end{itemize}
Hence $\theta_{\V,\W,\chi_\V,\chi_\W,\varPsi}(\varPi)$ is nonzero if and only if there exist $\varphi_{1,v}, \varphi_{2,v} \in \omega_{\V_v,\W_v,\chi_{\V,v},\chi_{\W,v},\varPsi_v}$ and $f_{1,v}, f_{2,v} \in \varPi_v$ such that
\[
 \ZZ(\varphi_{1,v}, \varphi_{2,v}, f_{1,v}, f_{2,v}) \ne 0
\]
for all $v$.

\subsection{Arthur's multiplicity formula}

In this subsection, we review Arthur's multiplicity formula for unitary groups \cite{mok,kmsw}, which is a key ingredient in the proof of Theorem \ref{t:main}\eqref{item:main1}.
Let $\F$ be a number field and $\E$ a quadratic extension of $\F$.
Let $\overline{\F}$ and $\overline{\F}_v$ be algebraic closures of $\F$ and $\F_v$, respectively, and fix an embedding $\overline{\F} \hookrightarrow \overline{\F}_v$ over $\F$ for each place $v$ of $\F$.
We also fix an embedding $\E \hookrightarrow \overline{\F}$ over $\F$, which determines an embedding $\E \hookrightarrow \overline{\F}_v$ for each place $v$ of $\F$ and hence a distinguished place $\tilde{v}$ of $\E$ above $v$.
If $v$ is split in $\E$, we identify $\E_v$ with $\F_v \times \F_v$ so that $\tilde{v}$ corresponds to the composition of the natural embedding $\E \hookrightarrow \E_v$ with the first projection $\E_v \rightarrow \F_v$.

Let $\V$ be an $n$-dimensional $\varepsilon$-Hermitian space over $\E$.
Then Arthur's endoscopic classification gives a decomposition of the automorphic discrete spectrum into near equivalence classes of representations:
\[
 L^2_{\mathrm{disc}}(\U(\V)(\F) \backslash \U(\V)(\A)) = \bigoplus_\varPhi L^2_\varPhi(\U(\V)),
\]
where $\varPhi$ runs over global $A$-parameters for $\U_n$, which is a formal unordered finite direct sum of the form
\[
 \varPhi = \bigoplus_i \varPhi_i \boxtimes S_{d_i},
\]
where
\begin{itemize}
\item $\varPhi_i$ is an irreducible conjugate-self-dual cuspidal automorphic representation of $\GL_{n_i}(\A_\E)$ with sign $(-1)^{n-d_i}$;
\item $S_{d_i}$ is the unique $d_i$-dimensional irreducible representation of $\SL_2(\C)$;
\item $(\varPhi_i, d_i) \ne (\varPhi_j, d_j)$ if $i \ne j$;
\item $\sum_i n_i d_i = n$.
\end{itemize}
Moreover, the multiplicity of each irreducible representation in $L^2_\varPhi(\U(\V))$ can be described as follows.

For each place $v$ of $\F$, we regard the localization $\varPhi_v = \bigoplus_i \varPhi_{i,v} \boxtimes S_{d_i}$ of $\varPhi$ at $v$ (where $\varPhi_{i,v}$ is viewed as a representation of $L_{\E_{\tilde{v}}}$ via the local Langlands correspondence) as a local $A$-parameter $\varPhi_v : L_{\F_v} \times \SL_2(\C) \rightarrow {}^L \U_n$ for $\U_n$.
Let $S_{\varPhi_v}$ be the local component group of $\varPhi_v$.
Recall that the local $A$-packet $\Pi_{\varPhi_v}(\U(\V_v))$ consists of semisimple representations of $\U(\V_v)$ of finite length.
We fix a global Whittaker datum, and with respect to its localization at $v$, we denote by $\sigma(\varPhi_v, \eta_v)$ the representation in $\Pi_{\varPhi_v}(\U(\V_v))$ associated to $\eta_v \in \widehat{S}_{\varPhi_v}$.
Let $S_\varPhi$ be the global component group of $\varPhi$, which is defined formally as a free $\Z/2\Z$-module
\[
 S_\varPhi = \bigoplus_i (\Z/2 \Z) e_i, 
\]
where $e_i$ corresponds to $\varPhi_i \boxtimes S_{d_i}$, and which is equipped with a natural homomorphism $S_\varPhi \rightarrow S_{\varPhi_v}$ for each $v$.
This gives rise to a compact group $S_{\varPhi, \A} = \prod_v S_{\varPhi_v}$ equipped with the diagonal map $\Delta : S_\varPhi \rightarrow S_{\varPhi, \A}$.
We denote by $\widehat{S}_{\varPhi, \A}$ the group of continuous characters of $S_{\varPhi, \A}$.
For any $\eta = \bigotimes_v \eta_v \in \widehat{S}_{\varPhi, \A}$, we may form a representation
\[
 \sigma(\varPhi,\eta) = \bigotimes_v \sigma(\varPhi_v,\eta_v)
\]
of $\U(\V)(\A)$.
Finally, let $\epsilon_\varPhi$ be the character of $S_\varPhi$ defined by \cite[(2.5.5)]{mok}.
Then Arthur's multiplicity formula \cite[Theorem* 1.7.1]{kmsw} says that 
\begin{equation}
\label{eq:amf}
 L^2_\varPhi(\U(\V)) \cong \bigoplus_\eta \sigma(\varPhi,\eta),  
\end{equation}
where $\eta$ runs over elements in $\widehat{S}_{\varPhi, \A}$ such that $\eta \circ \Delta = \epsilon_\varPhi$. 

We can describe the character $\epsilon_\varPhi$ more explicitly as follows.

\begin{lem}
\label{l:epsilon}
We have
\[
 \epsilon_\varPhi(e_i) = \prod_{j \ne i} \epsilon(\tfrac{1}{2}, \varPhi_i \times \varPhi_j^\vee)^{\min\{ d_i, d_j \}}, 
\]
where $\epsilon(s, \varPhi_i \times \varPhi_j^\vee)$ is the global $\epsilon$-factor of the pair $(\varPhi_i, \varPhi_j^\vee)$.
In particular, $\epsilon_\varPhi$ is trivial if $d_i=1$ for all $i$.
\end{lem}

\begin{proof}
The character $\epsilon_\varPhi$ is explicated in \cite[Proposition-Definition 8.3.7]{chenevier-lannes} in the case of orthogonal and symplectic groups.
We can apply the same argument to the case of unitary groups, noting that $\epsilon(\tfrac{1}{2}, \varPhi_i \times \varPhi_j^\vee) = 1$ if $\varPhi_i$ and $\varPhi_j$ have the same sign (see \cite[Theorem 2.5.4]{mok}).
\end{proof}

\subsection{Conjugate-self-dual characters}

In this subsection, we collect some results on conjugate-self-dual characters which we will use in the proof of Theorem \ref{t:main}\eqref{item:main1}.
Let $F$ be a local field of characteristic zero and $E$ an \'etale quadratic algebra over $F$.
Let $\psi$ be a nontrivial additive character of $F$ and define a nontrivial additive character $\psi^E_2$ of $E$ by $\psi^E_2(x) = \psi(\Tr_{E/F}(\delta x))$.
Let $\chi$ be a character of $E^\times$.
Then $\chi$ is conjugate-self-dual if and only if $\chi$ is trivial on $\mathrm{N}_{E/F}(E^\times)$.
Also, if $E \ne F \times F$, then $\chi$ is conjugate-orthogonal (resp.~conjugate-symplectic) if and only if $\chi|_{F^\times} = \mathbbm{1}$ (resp.~$\chi|_{F^\times} =\omega_{E/F}$).
We consider the value of the $\epsilon$-factor $\epsilon(s, \chi, \psi_2^E)$ at $s = \frac{1}{2}$.

\begin{lem}
\label{l:csd-char-epsilon}
Let $\chi$ be a conjugate-self-dual character of $E^\times$.
\begin{enumerate}
\item
\label{item:csd-char-epsilon1}
 If $E = F \times F$, then we have 
\[
 \epsilon(\tfrac{1}{2}, \chi, \psi_2^E) = 1.
\]
\item 
\label{item:csd-char-epsilon2}
If $E \ne F \times F$, then we have
\[
 \epsilon(\tfrac{1}{2}, \chi, \psi_2^E) = \pm 1.
\]
If further $\chi$ is conjugate-orthogonal, then we have
\[
 \epsilon(\tfrac{1}{2}, \chi, \psi_2^E) = 1.
\]
\item
\label{item:csd-char-epsilon3}
Suppose that $F = \R$ and $E = \C$.
Write
\[
 \chi(z) = \left( \frac{z}{\sqrt{z \bar{z}}} \right)^{2 \kappa}
\]
for some $\kappa \in \frac{1}{2} \Z$.
Assume further that $\delta = \sqrt{-1}$ and $\psi(x) = e^{- 2 \pi \sqrt{-1} x}$, so that $\psi_2^E(z) = e^{2 \pi (z - \bar{z})}$.
If $\kappa \in \Z$, then we have
\[
 \epsilon(\tfrac{1}{2}, \chi, \psi_2^E) = 1.
\]
If $\kappa \notin \Z$, then we have
\[
 \epsilon(\tfrac{1}{2}, \chi, \psi_2^E) = 
 \begin{cases}
  +1 & \text{if $\kappa > 0$;} \\
  -1 & \text{if $\kappa < 0$.}
 \end{cases}
\]
\end{enumerate}
\end{lem}

\begin{proof}
If $F = E \times E$, then we may write $\chi = \chi_0 \boxtimes \chi_0^{-1}$ and $\psi_2^E = \psi_0 \boxtimes \psi_0^{-1}$ for some characters $\chi_0$ and $\psi_0$ of $F^\times$ and $F$, respectively.
Then we have
\[
 \epsilon(\tfrac{1}{2}, \chi, \psi_2^E)
 = \epsilon(\tfrac{1}{2}, \chi_0, \psi_0) \cdot 
 \epsilon(\tfrac{1}{2}, \chi_0^{-1}, \psi_0^{-1})
 = 1, 
\]
so that \eqref{item:csd-char-epsilon1} follows.
For \eqref{item:csd-char-epsilon2} and \eqref{item:csd-char-epsilon3}, see \cite[Propositions 5.1 and 5.2]{ggp1} and \cite[Proposition 2.1]{ggp2}, respectively. 
\end{proof}

Let $F$ be a nonarchimedean local field of characteristic zero and $E$ a quadratic extension of $F$.
We prove the existence of a conjugate-symplectic character $\chi$ of $E^\times$ with a prescribed value of $\epsilon(\tfrac{1}{2}, \chi, \psi^E_2)$.

\begin{lem}
\label{l:csd-char-local}
Assume that either
\begin{itemize}
\item $E$ is unramified over $F$; or
\item the residual characteristic of $F$ is odd and $E$ is ramified over $F$.
\end{itemize}
Then there exists a conjugate-symplectic character $\chi$ of $E^\times$ such that
\[
 \epsilon(\tfrac{1}{2}, \chi, \psi^E_2) = 1.
\]
\end{lem}

\begin{proof}
If $E$ is unramified over $F$, then the assertion follows from \cite[Proposition 3.1]{ggp2}.
Hence we may assume that the residual characteristic of $F$ is odd and $E$ is ramified over $F$.
Then the order of $\psi \circ \Tr_{E/F}$ is odd, so that the order of $\psi^E_2$ is even, which we write as $2a$.
Let $\mathfrak{o}_F$ (resp.~$\mathfrak{o}_E$) be the maximal compact subring of $F$ (resp.~$E$), $\mathfrak{p}_F$ (resp.~$\mathfrak{p}_E$) the maximal ideal of $\mathfrak{o}_F$ (resp.~$\mathfrak{o}_E$), and $\varpi_F$ (resp.~$\varpi_E$) a uniformizer of $\mathfrak{o}_F$ (resp.~$\mathfrak{o}_E$).
We may assume that $\varpi_E^2 = \varpi_F$, so that $\Tr_{E/F}(\varpi_E) = 0$.
Since $\omega_{E/F}|_{1 + \mathfrak{p}_F} = \mathbbm{1}$ and $\mathfrak{o}_F^\times / (1 + \mathfrak{p}_F) \cong \mathfrak{o}_E^\times / (1 + \mathfrak{p}_E)$, there are precisely two conjugate-symplectic characters of $E^\times$ of conductor $1$.
Indeed, such a character $\chi$ is given by 
\[
 \chi|_{F^\times} = \omega_{E/F}, \quad
 \chi|_{1 + \mathfrak{p}_E} = \mathbbm{1}, \quad
 \chi(\varpi_E) = \zeta
\]
for some square root $\zeta$ of $\omega_{E/F}(\varpi_F)$.
Then we have 
\[
 \epsilon(\tfrac{1}{2}, \chi, \psi^E_2) = \chi(\varpi_E^{2a+1}) \cdot \frac{\mathcal{I}}{|\mathcal{I}|}, 
\]
where 
\[
 \mathcal{I} = \int_{\mathfrak{o}_E^\times} \chi(x)^{-1} \psi^E_2(\varpi_E^{-2a-1} x) \, dx.
\]
Note that $\chi(\varpi_E^{2a+1}) = \zeta \cdot \omega_{E/F}(\varpi_F^a)$ but that $\mathcal{I}$ does not depend on $\zeta$ since
\begin{align*}
 \mathcal{I} & = \operatorname{vol}(1 + \mathfrak{p}_E) \cdot \sum_{x \in \mathfrak{o}_E^\times/ (1 + \mathfrak{p}_E)} \chi(x)^{-1} \psi^E_2(\varpi_E^{-2a-1} x)  \\
 & = \operatorname{vol}(1 + \mathfrak{p}_E) \cdot \sum_{x \in \mathfrak{o}_F^\times/ (1 + \mathfrak{p}_F)} \omega_{E/F}(x)^{-1} \psi(2 \delta_0 x),
\end{align*}
where $\delta_0 = \varpi_E^{-2a-1} \delta \in F^\times$.
Hence we can choose $\zeta$ so that $\epsilon(\tfrac{1}{2}, \chi, \psi^E_2) = 1$.
\end{proof}

Let $\F$ be a number field and $\E$ a quadratic extension of $\F$.
Let $\Sigma$ be the set of places $v$ of $\F$ such that $\E_v \ne \F_v \times \F_v$.
We globalize local conjugate-self-dual characters to a global conjugate-self-dual character.

\begin{lem}
\label{l:csd-char-global}
For each $v \in \Sigma$, let $\chi_v$ be a conjugate-orthogonal (resp.~conjugate-symplectic) character of $\E_v^\times$.
Assume that $\chi_v$ is unramified for almost all $v \in \Sigma$.
Then there exists a conjugate-orthogonal (resp.~conjugate-symplectic) character $\chi_0$ of $\A_\E^\times/\E^\times$ such that
\[
 \chi_{0,v} = \chi_v
\]
for all $v \in \Sigma$.
\end{lem}

\begin{proof}
We may reduce the conjugate-symplectic case to the conjugate-orthogonal case by taking a conjugate-symplectic character $\chi'$ of $\A_\E^\times / \E^\times$ and applying the lemma to the character $\chi_v \cdot \chi_v'$ of $\E_v^\times$ for $v \in \Sigma$.
To treat the conjugate-orthogonal case, we consider an anisotropic torus
\[
 T = {\Res}_{E/F}(\mathbb{G}_m) / \mathbb{G}_m
\]
over $\F$.
For each $v \in \Sigma$, let $\nu_v$ be a character of $T_v$.
Assume that $\nu_v$ is unramified for almost all $v \in \Sigma$.
Then we may form a character $\nu_\Sigma = \bigotimes_{v \in \Sigma} \nu_v$ of $T_\Sigma = \prod_{v \in \Sigma} T_v$.
Since $T_\Sigma$ is compact, the image of the natural continuous embedding
\[
 T_\Sigma \longhookrightarrow T(\A) / T(\F)
\]
is closed. 
Hence we may extend $\nu_\Sigma$ to a character $\nu_0$ of $T(\A) / T(\F)$, so that 
\[
 \nu_{0,v} = \nu_v
\]
for all $v \in \Sigma$.
This completes the proof.
\end{proof}

\subsection{Global-to-local argument}

We now prove Theorem \ref{t:main}\eqref{item:main1}.
Thus we consider the theta lifting from $\U(p,q)$ to $\U(r,s)$ with $p+q=n$ and $r+s=m$ in the case
\[
 m>n.
\]
Let $\pi$ be a discrete series representation of $\U(p,q)$ with Harish-Chandra parameter $\lambda$.
We assume that its theta lift $\theta_{r,s}(\pi)$ to $\U(r,s)$ relative to $(\chi_V,\chi_W,\psi)$ is nonzero, where we take the data $(\chi_V,\chi_W,\psi)$ given in \S \ref{ss:main-theorem-setup}.
Then by Lemma \ref{l:local-theta-xyzw}, we may write 
\[
 \lambda = (\alpha_1, \dots, \alpha_x, \beta_1, \dots, \beta_y, \gamma_1, \dots, \gamma_z, \delta_1, \dots, \delta_w) + \bigg( \frac{m_0}{2}, \dots, \frac{m_0}{2} \bigg)
\]
with 
\begin{itemize}
\item $\alpha_i, \gamma_j > 0$ and $\beta_i, \delta_j \le 0$;
\item $x+y = p$ and $z+w = q$;
\item $x+w \le r$ and $z+y \le s$.
\end{itemize}
As in \S \ref{sss:local-theta-real}, we define an $L$-parameter 
\[
 \phi = \chi_{\kappa_1} \oplus \dots \oplus \chi_{\kappa_n}
\]
for $\U_n$ such that $\pi$ belongs to the $L$-packet $\Pi_\phi(\U(p,q))$.
Let $\eta$ be the character of $S_\phi$ associated to $\pi$ as in \S \ref{sss:packets-real-ds}, so that $\pi = \pi(\phi, \eta)$.

To prove Theorem \ref{t:main}\eqref{item:main1}, we appeal to a global-to-local argument and derive the information about $\theta_{r,s}(\pi)$ from the knowledge of $\theta_{r,s}(\pi_+)$, where $\pi_+$ is a discrete series representation of $\U(p,q)$ with sufficiently regular infinitesimal character.
More precisely, we assume that the Harish-Chandra parameter $\lambda_+$ of $\pi_+$ is of the form 
\[
 \lambda_+ = \lambda + (\underbrace{t, \dots, t}_x, \underbrace{-t, \dots, -t}_y, \underbrace{t, \dots, t}_z, \underbrace{-t, \dots, -t}_w)
\]
for some positive integer 
\[
 t \ge \frac{m-n+1}{2}.
\]
As in \S \ref{sss:local-theta-real}, we define an $L$-parameter 
\[
 \phi_+ = \chi_{\kappa_{+,1}} \oplus \dots \oplus \chi_{\kappa_{+,n}}
\]
for $\U_n$ such that $\pi_+$ belongs to the $L$-packet $\Pi_{\phi_+}(\U(p,q))$.
Then we have
\[
 \pi_+ = \pi(\phi_+, \eta),
\]
where $\eta$ is viewed as a character of $S_{\phi_+}$ via the canonical isomorphism $S_{\phi_+} \cong S_\phi$.

To simplify the argument, we also need an auxiliary irreducible representation $\pi_0$ of $\U(n,0)$ with Harish-Chandra parameter $(\kappa_{0,1}, \dots, \kappa_{0,n})$ such that 
\[
 \kappa_{0,1} > \dots > \kappa_{0,n} > \frac{m_0+m-n+1}{2}.
\]
Define an $L$-parameter $\phi_0$ for $\U_n$ by 
\[
 \phi_0 = \chi_{\kappa_{0,1}} \oplus \dots \oplus \chi_{\kappa_{0,n}}.
\]
Then $\pi_0$ belongs to the $L$-packet $\Pi_{\phi_0}(\U(n,0))$.
Let $\eta_0$ be the character of $S_{\phi_0}$ associated to $\pi_0$ as in \S \ref{sss:packets-real-ds}, so that $\pi_0 = \pi(\phi_0, \eta_0)$.

We now globalize everything in sight.
Let $\F$ be a real quartic field and $\E$ a totally imaginary quadratic extension of $\F$ such that $\E_v = \F_v \times \F_v$ for all places $v$ of $\F$ above $2$.
Let $v_0,v_1,v_2,v_3$ be the four real places of $\F$.
Fix an element $\delta \in \E^\times$ with $\Tr_{\E/\F}(\delta) = 0$ and a nontrivial additive character $\varPsi$ of $\A/\F$ such that 
\begin{itemize}
\item $\delta$ belongs to the $(\F_{v_i}^\times)^2$-orbit of $\sqrt{-1}$ for $i=0,1,2,3$;
\item $\varPsi_{v_i}$ belongs to the $(\F_{v_i}^\times)^2$-orbit of $\psi$ for $i=0,1,2,3$.
\end{itemize}
We will take the global Whittaker datum determined by $\delta$ and $\varPsi$ as in \S \ref{ss:whit}.
Let $\W$ be the $n$-dimensional anisotropic skew-Hermitian space over $\E$ such that
\begin{itemize}
\item the signature of $\W_{v_i}$ is $(p,q)$ for $i=0,1$;
\item the signature of $\W_{v_i}$ is $(n,0)$ for $i=2,3$;
\item $\epsilon(\W_v) = 1$ for all nonarchimedean places $v$ of $\F$.
\end{itemize}
Similarly, let $\V$ be the $m$-dimensional anisotropic Hermitian space over $\E$ such that
\begin{itemize}
\item the signature of $\V_{v_i}$ is $(r,s)$ for $i=0,1$;
\item the signature of $\V_{v_i}$ is $(m,0)$ for $i=2,3$;
\item $\epsilon(\V_v) = 1$ for all nonarchimedean places $v$ of $\F$.
\end{itemize}
Note that such spaces $\W$ and $\V$ exist since $\prod_v \epsilon(\W_v) = \prod_v \epsilon(\V_v) = 1$.
By Lemma \ref{l:csd-char-global}, we may take two (unitary) characters $\chi_\V, \chi_\W$ of $\A_\E^\times/\E^\times$ such that
\begin{alignat*}{2}
 \chi_\V|_{\A^\times} & = \omega_{\E/\F}^m, \quad & \chi_{\V, v_i} & = \chi_V, \\
 \chi_\W|_{\A^\times} & = \omega_{\E/\F}^n, \quad & \chi_{\W, v_i} & = \chi_W
\end{alignat*}
for $i=0,1,2,3$.
Similarly, by Lemmas \ref{l:csd-char-local} and \ref{l:csd-char-global}, we may take conjugate-self-dual characters $\chi_1, \dots, \chi_n$ of $\A_\E^\times/\E^\times$ with sign $(-1)^{n-1}$ satisfying the following conditions:
\begin{itemize}
\item $\chi_{i,v_0} = \chi_{\kappa_i}$ for all $i$;
\item $\chi_{i,v_1} = \chi_{\kappa_{+,i}}$ for all $i$;
\item $\chi_{i,v_2} = \chi_{i,v_3} = \chi_{\kappa_{0,i}}$ for all $i$;
\item if $m \equiv n \bmod 2$, then 
\begin{equation}
\label{eq:assumption-epsilon}
 \epsilon(\tfrac{1}{2}, \chi_{i,v} \chi_{\V,v}^{-1}, \varPsi_{2,v}^\E) = 1 
\end{equation}
for all nonarchimedean places $v$ of $\F$ such that $\E_v \ne \F_v \times \F_v$, where $\varPsi_{2,v}^\E$ is the character of $\E_v$ given by $\varPsi_{2,v}^\E(x) = \varPsi_v(\Tr_{\E_v/\F_v}(\delta x))$.
\end{itemize}
In particular, $\chi_1, \dots, \chi_n, \chi_\V$ are pairwise distinct.

Define a global $A$-parameter $\varPhi$ for $\U_n$ by
\[
 \varPhi = \chi_1 \oplus \dots \oplus \chi_n,
\]
so that
\[
 \varPhi_{v_0} = \phi, \quad \varPhi_{v_1} = \phi_+, \quad \varPhi_{v_2} = \varPhi_{v_3} = \phi_0.
\]
Let $S_\varPhi$ be the global component group of $\varPhi$, which is defined formally as a free $\Z/2\Z$-module
\[
 S_\varPhi = (\Z/2\Z) e_1 \oplus \dots \oplus (\Z/2 \Z) e_n,
\]
where $e_i$ corresponds to $\chi_i$.
For each place $v$ of $\F$, let $S_{\varPhi_v}$ be the local component group of $\varPhi_v$ equipped with a natural homomorphism $S_\varPhi \rightarrow S_{\varPhi_v}$.
Note that this homomorphism is an isomorphism for $v = v_0,v_1,v_2,v_3$.
We denote by $e_{i,v}$ the image of $e_i$ in $S_{\varPhi_v}$.
Recall the compact group $S_{\varPhi, \A} = \prod_v S_{\varPhi_v}$ equipped with the diagonal map $\Delta : S_\varPhi \rightarrow S_{\varPhi, \A}$.
Define a continuous character $\eta = \bigotimes_v \eta_v$ of $S_{\varPhi,\A}$ by
\begin{itemize}
\item $\eta_{v_0} = \eta_{v_1} = \eta$;
\item $\eta_{v_2} = \eta_{v_3} = \eta_0$;
\item $\eta_v = \mathbbm{1}$ for all nonarchimedean places $v$ of $\F$.
\end{itemize}
Note that $\eta_v(e_{1,v} + \dots + e_{n,v}) = \epsilon(\W_v)$ for all $v$.
Let 
\[
 \varPi_v = \pi(\varPhi_v, \eta_v)
\]
be the irreducible tempered representation in the local $L$-packet $\Pi_{\varPhi_v}(\U(\W_v))$ associated to $\eta_v$.
Then we may form an irreducible representation $\varPi = \bigotimes_v \varPi_v$ of $\U(\W)(\A)$.
Since
\[
 (\eta \circ \Delta) (e_i) = \prod_v \eta_v(e_{i,v}) = 1
\]
for all $i$, it follows from Arthur's multiplicity formula \eqref{eq:amf} that $\varPi$ is automorphic.
Moreover, the partial standard $L$-function
\[
 L^S(s, \varPi, \chi_\V^{-1}) = L^S(s, \chi_1 \chi_{\V}^{-1}) \cdots L^S(s, \chi_n \chi_{\V}^{-1})
\]
is holomorphic and nonzero at $s = \frac{1}{2}(m-n+1)$.

We consider the global theta lift $\varSigma = \theta_{\V, \W, \chi_\V, \chi_\W, \varPsi}(\varPi)$ to $\U(\V)(\A)$.
Recall that the local theta lift $\varSigma_v = \theta_{\V_v, \W_v, \chi_{\V,v}, \chi_{\W,v}, \varPsi_v}(\varPi_v)$ to $\U(\V_v)$ is nonzero for $v = v_0$ by assumption and for $v \ne v_0$ by Lemmas \ref{l:local-theta-1}, \ref{l:local-theta-2}, \ref{l:local-theta-3}.
In the next section, we will show that there exist $\varphi_{1,v}, \varphi_{2,v} \in \omega_{\V_v,\W_v,\chi_{\V,v},\chi_{\W,v},\varPsi_v}$ and $f_{1,v}, f_{2,v} \in \varPi_v$ such that
\begin{equation}
\label{eq:nonvanishing-zeta}
 \ZZ(\varphi_{1,v}, \varphi_{2,v}, f_{1,v}, f_{2,v}) \ne 0
\end{equation}
for all $v$.
As explained in \S \ref{ss:global-theta}, this implies that $\varSigma$ is nonzero.
Thus we obtain an irreducible automorphic representation $\varSigma = \bigotimes_v \varSigma_v$ of $\U(\V)(\A)$.

Finally, we derive the information about $\varSigma_{v_0} = \theta_{r,s}(\pi)$ from the knowledge of $\varSigma_v$ for $v \ne v_0$ and Arthur's multiplicity formula.
Define a global $A$-parameter $\varPhi'$ for $\U_m$ by
\[
 \varPhi' = \chi_1 \chi_\V^{-1} \chi_\W \oplus \dots \oplus \chi_n \chi_\V^{-1} \chi_\W \oplus (\chi_\W \boxtimes S_{m-n}).
\]
Let $S_{\varPhi'}$ be the global component group of $\varPhi'$, which is defined formally as a free $\Z/2\Z$-module
\[
 S_{\varPhi'} = (\Z/2\Z) e'_1 \oplus \dots \oplus (\Z/2 \Z) e'_n \oplus (\Z/2 \Z) e'_0,
\]
where $e'_i$ corresponds to $\chi_1 \chi_\V^{-1} \chi_\W$ (resp.~$\chi_\W \boxtimes S_{m-n}$) if $i \ne 0$ (resp.~$i=0$).
For each place $v$ of $\F$, let $S_{\varPhi'_v}$ be the local component group of $\varPhi'_v$ equipped with a natural homomorphism $S_{\varPhi'} \rightarrow S_{\varPhi'_v}$.
We denote by $e'_{i,v}$ the image of $e'_i$ in $S_{\varPhi'_v}$.
By Lemmas \ref{l:local-theta-2} and \ref{l:local-theta-3}, $\varSigma$ occurs in the near equivalence class $L^2_{\varPhi'}(\U(\V))$.
Hence $\varSigma_v$ belongs to the local $A$-packet $\Pi_{\varPhi'_v}(\U(\V_v))$ for all $v$.
Since $\Pi_{\varPhi'_v}(\U(\V_v))$ is multiplicity-free, we may associate to $\varSigma_v$ a character $\eta'_v$ of $S_{\varPhi'_v}$.
Then it follows from Arthur's multiplicity formula \eqref{eq:amf} and Lemma \ref{l:epsilon} that
\[
 \prod_v \eta_v'(e_{i,v}') =
 \begin{cases}
  \epsilon(\tfrac{1}{2}, \chi_i \chi_{\V}^{-1}) & \text{if $i \ne 0$;} \\
  \epsilon(\tfrac{1}{2}, \varPi, \chi_{\V}^{-1}) & \text{if $i = 0$.}
 \end{cases}
\]
However, it follows from Lemma \ref{l:csd-char-epsilon} and \eqref{eq:assumption-epsilon} that
\[
 \epsilon(\tfrac{1}{2}, \chi_i \chi_{\V}^{-1}) = \prod_v \epsilon(\tfrac{1}{2}, \chi_{i,v} \chi_{\V,v}^{-1}, \varPsi_{2,v}^\E) = 1,
\]
so that 
\[
 \epsilon(\tfrac{1}{2}, \varPi, \chi_{\V}^{-1})
 = \epsilon(\tfrac{1}{2}, \chi_1 \chi_{\V}^{-1}) \cdots \epsilon(\tfrac{1}{2}, \chi_n \chi_{\V}^{-1}) = 1.
\]
Hence we have
\[
  \prod_v \eta_v'(e_{i,v}') = 1
\]
for all $i$.
On the other hand, by Lemmas \ref{l:local-theta-2} and \ref{l:local-theta-3}, we have $\eta_v' = \mathbbm{1}$ for all nonarchimedean places $v$ of $\F$.
Since $\eta_{v_2}' = \eta_{v_3}'$, we conclude that
\begin{equation}
\label{eq:main} 
 \eta_{v_0}'(e_{i,v_0}) = \eta_{v_1}'(e_{i,v_1})
\end{equation}
for all $i$.

\subsection{Completion of the proof}

We have shown that
\[
 \theta_{r,s}(\pi) = \sigma(\phi', \eta'), 
\]
where $\phi' = \varPhi'_{v_0}$ is the $A$-parameter for $\U_m$ and $\eta' = \eta'_{v_0}$ is the character of $S_{\phi'}$ as in the previous subsection.
More explicitly, we have
\[
 \phi' = \chi_{\kappa_1} \chi_V^{-1} \chi_W \oplus \dots \oplus \chi_{\kappa_n} \chi_V^{-1} \chi_W \oplus (\chi_W \boxtimes S_{m-n})
\]
by construction of $\varPhi'$ and 
\[
 \eta'(e_i') = \zeta_i \times 
 \begin{cases}
  \eta(e_i) & \text{if $i \ne 0$;} \\
  \epsilon(V) \cdot \epsilon(W) & \text{if $i = 0$}
 \end{cases}
\]
by \eqref{eq:main} and Lemma \ref{l:local-theta-1}, where we write $e_i = e_{i,v_0}$, $e'_i = e'_{i,v_0}$ for brevity and define $\zeta_i = \pm 1$ as in Lemma \ref{l:local-theta-1}.
Then we can apply the argument in the proof of Lemma \ref{l:local-theta-1} to deduce that 
\[
 \sigma(\phi', \eta') = A_\q(\lambda'),
\]
where $\q$ and $\lambda'$ are as given in Theorem \ref{t:main}\eqref{item:main1}.
This completes the proof of Theorem \ref{t:main}\eqref{item:main1}.

\begin{rem}
Since $\eta'$ is a character of $S_{\phi'}$, we must have 
\[
 \eta'(e_0') = \eta'(e_{i_0}')
\]
if $\kappa_{i_0} = \frac{m_0}{2}$ and $m-n=1$ (in which case $S_{\phi'} \ne \widetilde{S}_{\phi'}$).
This can also be proved directly using Theorem \ref{t:nonvanishing} and Lemma \ref{l:eta'}.
\end{rem}

\begin{rem}
In the global-to-local argument, we can also use theta lifts for $p$-adic unitary groups instead of those for real unitary groups.
For this, we need to modify the argument as follows.
\begin{itemize}
\item 
Instead of \cite{li90}, we use a result of Atobe--Gan \cite{atobe-gan} to describe some theta lifts for $p$-adic unitary groups explicitly.
\item 
Instead of \cite{mr19}, we use an analog of a result of M{\oe}glin \cite{moeglin06} (see also \cite[\S 6]{xu}) for $p$-adic unitary groups to describe the representations in some $A$-packets explicitly.
\end{itemize}
\end{rem}

\section{Nonvanishing of integrals of matrix coefficients}
\label{s:nonvanish}

In this section, we prove \eqref{eq:nonvanishing-zeta}, which completes the proof of Theorem \ref{t:main}.

\subsection{Doubling zeta integrals}

Let $F$ be a local field of characteristic zero and $E$ an \'etale quadratic algebra over $F$.
Let $W$ be an $n$-dimensional skew-Hermitian space over $E$.
We equip $W^\square = W \oplus W$ with the skew-Hermitian form given by
\[
 \langle (w_1,w_1'), (w_2, w_2') \rangle_{W^\square} = \langle w_1, w_2 \rangle_W - \langle w_1', w_2' \rangle_W.
\]
Then we have a natural embedding
\[
 i: \U(W) \times \U(W) \longhookrightarrow \U(W^\square).
\]
Define a complete polarization $W^\square = W^\triangle \oplus W^\bigtriangledown$ by 
\[
 W^\triangle = \{ (w,w) \, | \, w \in W \}, \quad
 W^\bigtriangledown = \{ (w,-w) \, | \, w \in W \}.
\]
Let $P$ be the maximal parabolic subgroup of $\U(W^\square)$ stabilizing $W^\triangle$.
For $s \in \C$ and a character $\chi$ of $E^\times$, we write 
\[
 I(s, \chi) = \Ind^{\U(W^\square)}_P(\chi |\cdot|_E^s),
\]
where $\Ind^{\U(W^\square)}_P$ denotes the normalized parabolic induction and $\chi |\cdot|_E^s$ is viewed as a character of $P$ via the natural homomorphism
\[
 P \longrightarrow \GL(W^\triangle) \overset{\det}\longrightarrow E^\times.
\]

Let $\pi$ be an irreducible tempered representation of $\U(W)$ and $(\cdot, \cdot) : \pi \times \pi \rightarrow \C$ an invariant Hermitian inner product.
We consider the zeta integral
\[
 Z(\FF_s, f_1, f_2) = \int_{\U(W)} \FF_s(i(g,1)) \overline{(\pi(g) f_1, f_2)} \, dg
\]
for $\FF_s \in I(s, \chi)$ and $f_1, f_2 \in \pi$.

\begin{lem}
\label{l:conv}
If $\Re(s) > -\frac{1}{2}$, then the integral $Z(\FF_s, f_1, f_2)$ is absolutely convergent.
\end{lem}

\begin{proof}
See \cite[Lemma 9.5]{gi1}, \cite[Lemma 7.2]{yamana}.
\end{proof}

\subsection{Integrals of matrix coefficients}
\label{ss:matcoeff-integral}

Let $V$ be an $m$-dimensional Hermitian space over $E$.
Recall the symplectic space $\W = V \otimes_E W$ over $F$ and take a complete polarization $\W = \X \oplus \Y$.
Then we may realize the Weil representation $\omega = \omega_{V,W,\chi_V,\chi_W,\psi}$ of $\U(V) \times \U(W)$ on $S(\X)$, where 
\begin{itemize}
\item 
if $F$ is archimedean, we denote by $\mathcal{S}(\X)$ the space of Schwartz functions on $\X$ and by $S(\X)$ the subspace of $\mathcal{S}(\X)$ consisting of functions which correspond to polynomials in the Fock model;
\item 
if $F$ is nonarchimedean, we denote by $S(\X)$ the space of locally constant compactly supported functions on $\X$.
\end{itemize}
Similarly, consider the symplectic space $\W^\square = V \otimes_E W^\square$ over $F$ and define a complete polarization $\W^\square = \W^\triangle \oplus \W^\bigtriangledown$ by
\[
 \W^\triangle = V \otimes_E W^\triangle, \quad
 \W^\bigtriangledown = V \otimes_E W^\bigtriangledown.
\]
Then we may realize the Weil representation $\Omega = \omega_{V,W^\square,\chi_V,\mathbbm{1},\psi}$ of $\U(V) \times \U(W^\square)$ on $S(\W^\bigtriangledown)$.
Moreover, there is a $\U(W) \times \U(W)$-equivariant isomorphism
\[
 T: (\omega \boxtimes (\bar{\omega} \otimes (\chi_V \circ \det)), S(\X) \otimes S(\X)) \overset{\cong}{\longrightarrow} (\Omega \circ i, S(\W^\bigtriangledown))
\]
such that 
\[
 T(\varphi_1 \otimes \bar{\varphi}_2)(0) = (\varphi_1, \varphi_2)
\]
for $\varphi_1, \varphi_2 \in S(\X)$, where $(\cdot,\cdot) : S(\X) \times S(\X) \rightarrow \C$ is an invariant Hermitian inner product.

On the other hand, we may define a $\U(W^\square)$-equivariant map
\[
 \FF: S(\W^\bigtriangledown) \longrightarrow I \left( \frac{m-n}{2}, \chi_V \right)
\]
by 
\[
 \FF(\varphi)(g) = (\Omega(g) \varphi)(0)
\]
for $\varphi \in S(\W^\bigtriangledown)$ and $g \in \U(W^\square)$.
Hence, if $m \ge n$, then we obtain a map
\[
 \ZZ : S(\X) \times S(\X) \times \pi \times \pi \longrightarrow \C
\]
defined by
\begin{align*}
 \ZZ(\varphi_1, \varphi_2, f_1, f_2) & = Z(\FF(T(\varphi_1 \otimes \bar{\varphi}_2)), f_1, f_2) \\
 & = \int_{\U(W)} (\omega(g) \varphi_1, \varphi_2) \overline{(\pi(g) f_1, f_2)} \, dg
\end{align*}
for $\varphi_1, \varphi_2 \in S(\X)$ and $f_1, f_2 \in \pi$, where the integral is absolutely convergent by Lemma \ref{l:conv}.
We also write 
\[
 \ZZ = \ZZ_{V,W,\chi_V,\chi_W,\psi}(\pi)
\]
to indicate the data we are using.
Obviously, for $\varphi_2 \in S(\X)$ and $f_2 \in \pi$, the map
\[
 \varphi_1 \otimes \bar{f}_1 \longmapsto \ZZ(\varphi_1, \varphi_2, f_1, f_2)
\]
defines an element in
\[
 \Hom_{\U(W)}(\omega \otimes \bar{\pi}, \C) \cong \Hom_{\U(W)}(\omega, \pi).
\]
Hence, if $\ZZ_{V,W,\chi_V,\chi_W,\psi}(\pi)$ is nonzero, then the theta lift $\theta_{V,W,\chi_V,\chi_W,\psi}(\pi)$ is nonzero.
In fact, we have the following converse.

\begin{prop}
\label{p:key}
Let $\pi$ be an irreducible tempered representation (resp.~a discrete series representation) of $\U(W)$ if $F$ is nonarchimedean or $E = F \times F$ (resp.~$F = \R$ and $E = \C$).
Assume that $m \ge n$.
Then $\ZZ_{V,W,\chi_V,\chi_W,\psi}(\pi)$ is nonzero if and only if $\theta_{V,W,\chi_V,\chi_W,\psi}(\pi)$ is nonzero.
\end{prop}

If $F$ is nonarchimedean or $E = F \times F$, then this proposition has been proved in \cite[Proposition 11.5]{gqt}, \cite[Lemma 8.6]{yamana}.
The rest of this section is devoted to the proof in the remaining case, which is the key technical innovation in this paper. 

\subsection{Inductive step}

From now on, we assume that $F=\R$ and $E=\C$.
Let $\pi$ be a discrete series representation of $G = \U(W)$.
We denote by $\hat{\pi}$ the unitary completion of $\pi$, i.e.~a unitary representation of $G$ on a Hilbert space $\HH$ such that $\pi$ is isomorphic to the $(\g,K)$-module on the space $\HH_K$ of $K$-finite vectors in $\HH$.
Here $\g$ is the complexified Lie algebra of $G$ and $K$ is the maximal compact subgroup of $G$ as in \S \ref{ss:ds}.
We regard $\ZZ_{V,W,\chi_V,\chi_W,\psi}(\pi)$ as a map on
\[
 S(\X) \times S(\X) \times \HH_K \times \HH_K.
\]
By abuse of notation, we also regard the Weil representation $\omega$ as a smooth representation on $\mathcal{S}(\X)$ equipped with the usual topology.

\begin{lem}
\label{l:cont}
The integral
\[
 \hat{\ZZ}(\varphi_1, \varphi_2, f_1, f_2) = \int_G (\omega(g) \varphi_1, \varphi_2) \overline{(\hat{\pi}(g) f_1, f_2)} \, dg 
\]
is absolutely convergent for $\varphi_1, \varphi_2 \in \mathcal{S}(\X)$ and $f_1, f_2 \in \HH$, and defines a separately continuous map
\[
 \hat{\ZZ} : \mathcal{S}(\X) \times \mathcal{S}(\X) \times \HH \times \HH \longrightarrow \C.
\]
\end{lem}

\begin{proof}
By \cite[Theorem 3.2]{li89}, the function $g \mapsto (\omega(g) \varphi_1, \varphi_2)$ is square-integrable, which implies the absolute convergence.
The separate continuity follows from the argument in the proof of \cite[Lemma 6.2]{li89}, which we include here for the convenience of the reader.
We have
\[
 |\hat{\ZZ}(\varphi_1, \varphi_2, f_1, f_2)|^2 \le C \cdot \int_G |(\hat{\pi}(g) f_1, f_2)|^2 \, dg = \frac{C}{\deg \hat{\pi}} \cdot \| f_1 \|^2 \cdot \| f_2 \|^2
\]
with 
\[
 C = \int_G |(\omega(g) \varphi_1, \varphi_2)|^2 \, dg, 
\]
where $\deg \hat{\pi}$ is the formal degree of $\hat{\pi}$ and $\| \cdot \|$ is the Hilbert norm on $\HH$.
This implies the separate continuity in the third and fourth variables.
Let $\{ \varphi_{1,i} \}_{i \ge 1}$ be a sequence converging to $\varphi_1$ in $\mathcal{S}(\X)$.
By \cite[Lemme 5]{weil1}, there exists $\varphi_0 \in \mathcal{S}(\X)$ such that 
\[
 |\varphi_{1,i}(x)| \le \varphi_0(x)
\]
for all $i \ge 1$ and $x \in \X$.
Fix $\varphi_2 \in \mathcal{S}(\X)$ and put 
\[
 \Phi(g) = \int_{\X} \varphi_0(x) |\omega(g^{-1}) \varphi_2(x)| \, dx, 
\]
so that 
\[
 |(\omega(g) \varphi_{1,i}, \varphi_2)| = \left| \int_{\X} \varphi_{1,i}(x) \overline{\omega(g^{-1})\varphi_2(x)} \, dx \right| \le \Phi(g)
\]
for all $i \ge 1$ and $g \in G$.
Note that $\Phi$ is square-integrable (see the proof of \cite[Theorem 3.2]{li89}).
Then we have
\[
 |\hat{\ZZ}(\varphi_{1,i}, \varphi_2, f_1, f_2)|^2 \le C' \cdot \int_G |(\omega(g) \varphi_{1,i}, \varphi_2)|^2 \, dg \le C' \cdot \int_G \Phi(g)^2 \, dg
\]
with 
\[
 C' = \int_G |(\hat{\pi}(g) f_1, f_2)|^2 \, dg = \frac{1}{\deg \hat{\pi}} \cdot \| f_1 \|^2 \cdot \| f_2 \|^2
\]
for all $i \ge 1$.
Hence the dominated convergence theorem implies the separate continuity in the first variable.
The proof for the second variable is similar.
\end{proof}

\begin{lem}
\label{l:gK-vs-L2}
Assume that $\ZZ_{V,W,\chi_V,\chi_W,\psi}(\pi) \ne 0$.
Let $f_1 \in \HH$ be a nonzero element.
Then there exist $\varphi_1, \varphi_2 \in \mathcal{S}(\X)$ and $f_2 \in \HH$ such that 
\[
 \hat{\ZZ}(\varphi_1, \varphi_2, f_1, f_2) \ne 0.
\]
\end{lem}

\begin{proof}
Fix $\varphi_{1,0}, \varphi_{2,0} \in S(\X)$ and $f_{1,0}, f_{2,0} \in \HH_K$ such that 
\[
 \hat{\ZZ}(\varphi_{1,0}, \varphi_{2,0}, f_{1,0}, f_{2,0}) \ne 0.
\]
Define a linear $G$-equivariant map $A: \mathcal{S}(\X) \rightarrow L^2(G)$ by
\[
 A(\varphi)(g) = (\omega(g) \varphi, \varphi_{2,0})
\]
and a continuous $G$-equivariant embedding $B: \HH \hookrightarrow L^2(G)$ by 
\[
 B(f)(g) = (\hat{\pi}(g) f, f_{2,0}),
\]
where $G$ acts on $L^2(G)$ by right translation.
Then it follows from the proof of \cite[Lemma 2.2]{li90} that $\operatorname{Im} B$ is contained in the closure of $\operatorname{Im} A$ in $L^2(G)$.
Hence we must have
\[
 \int_G A(\varphi_1)(g) \overline{B(f_1)(g)} \, dg \ne 0
\]
for some $\varphi_1 \in \mathcal{S}(\X)$.
This completes the proof.
\end{proof}

To prove Proposition \ref{p:key}, we proceed by induction on $\dim V - \dim W$ as in the proof of \cite[Proposition 5.4]{atobe}, where $V$ is fixed but $W$ and $\pi$ vary.
For this, we need a seesaw diagram
\[
\begin{tikzcd}
 \U(W') \arrow[rd,dash] & \U(V) \times \U(V) \arrow[ld,dash] \\
 \U(W) \times \U(W^\perp) \arrow[u,dash] & \U(V) \arrow[u,dash]
\end{tikzcd},
\]
where $V$ is an $m$-dimensional Hermitian space over $\C$, $W'$ is an $(n+1)$-dimensional skew-Hermitian space over $\C$, $W$ is an $n$-dimensional skew-Hermitian subspace of $W'$, and $W^\perp$ is the orthogonal complement of $W$ in $W'$.
Let $\pi$ and $\pi'$ be discrete series representations of $G = \U(W)$ and $G' = \U(W')$, respectively.
We denote by $\hat{\pi}$ and $\hat{\pi}'$ the unitary completions of $\pi$ and $\pi'$, respectively, and by $(\hat{\pi}'|_G)_{\disc}$ the discrete spectrum of the restriction of $\hat{\pi}'$ to $G$.

\begin{lem}
\label{l:seesaw}
Assume that $m > n$ and that $\hat{\pi}$ occurs in $(\hat{\pi}'|_G)_{\disc}$.
Then we have
\[
 \ZZ_{V,W',\chi_V,\chi_{W'},\psi}(\pi') \ne 0 \quad \Longrightarrow \quad \ZZ_{V,W,\chi_V,\chi_W,\psi}(\pi) \ne 0.
\]
\end{lem}

\begin{proof}
Consider the symplectic spaces
\[
 \W' = V \otimes_E W', \quad 
 \W = V \otimes_E W, \quad
 \W^\perp = V \otimes_E W^\perp
\]
over $F$, so that $\W' = \W \oplus \W^\perp$.
We may take complete polarizations
\[
 \W' = \X' \oplus \Y', \quad
 \W = \X \oplus \Y, \quad
 \W^\perp = \X^\perp \oplus \Y^\perp
\]
such that $\X' = \X \oplus \X^\perp$ and $\Y' = \Y \oplus \Y^\perp$.
Write 
\[
 \omega' = \omega_{V,W',\chi_V,\chi_{W'},\psi}, \quad
 \omega = \omega_{V,W,\chi_V,\chi_W,\psi}, \quad
 \omega^\perp = \omega_{V,W^\perp,\chi_V,\chi_{W'}\chi_W^{-1},\psi}.
\]
Then we have an identification
\[
 (\omega', S(\X')) = (\omega \boxtimes \omega^\perp, S(\X) \otimes S(\X^\perp))
\]
as representations of $\U(W) \times \U(W^\perp) \times \U(V)$.

Assume that $\ZZ_{V,W',\chi_V,\chi_{W'},\psi}(\pi') \ne 0$.
Let $\HH$ and $\HH'$ be the underlying Hilbert spaces of $\hat{\pi}$ and $\hat{\pi}'$, respectively.
By assumption, we may realize $\HH$ as a closed $G$-invariant subspace of $\HH'$.
Let $f_1' \in \HH$ be a nonzero element.
By Lemma \ref{l:gK-vs-L2}, there exist $\varphi_1', \varphi_2' \in \mathcal{S}(\X')$ and $f_2' \in \HH'$ such that
\[
 \int_{G'} (\omega'(g') \varphi_1', \varphi_2') \overline{(\hat{\pi}'(g') f_1', f_2')} \, dg' \ne 0.
\]
This integral is absolutely convergent and is equal to
\begin{align*}
 & \int_{G'/G} \int_G (\omega'(g'g) \varphi_1', \varphi_2') \overline{(\hat{\pi}'(g'g) f_1', f_2')} \, dg \, dg' \\
 & = \int_{G'/G} \int_G (\omega'(g) \varphi_1', \omega'(g'^{-1}) \varphi_2') \overline{(\hat{\pi}'(g) f_1', \hat{\pi}'(g'^{-1}) f_2')} \, dg \, dg'.
\end{align*}
Hence we have
\[
 \int_G (\omega'(g) \varphi_1', \varphi_3') \overline{(\hat{\pi}'(g) f_1', f_3')} \, dg \ne 0
\]
for some $\varphi_3' \in \mathcal{S}(\X')$ and $f_3' \in \HH'$.
We may assume that $f_3' \in \HH$ by replacing $f_3'$ by its orthogonal projection to $\HH$ if necessary.
Note that the function $g \mapsto (\hat{\pi}'(g) f_1', f_3')$ on $G$ is a matrix coefficient of $\hat{\pi}$.
On the other hand, for any $\varphi' \in \mathcal{S}(\X')$, the function $g \mapsto (\omega'(g) \varphi', \varphi'_3)$ on $G$ is square-integrable (see the proof of \cite[Theorem 3.2]{li89}).
From this and the argument in the proof of Lemma \ref{l:cont}, we can deduce that the map
\[
 \varphi' \longmapsto \int_G (\omega'(g) \varphi', \varphi'_3) \overline{(\hat{\pi}'(g) f_1', f_3')} \, dg 
\]
on $\mathcal{S}(\X')$ is continuous.
Hence we may assume that $\varphi_1' \in S(\X')$.
Similarly, we may assume that $\varphi_3' \in S(\X')$.
Then we may assume further that $\varphi_i' = \varphi_i \otimes \varphi_i^\perp$ with $\varphi_i \in S(\X)$ and $\varphi_i^\perp \in S(\X^\perp)$.
Since 
\[
 (\omega'(g) \varphi_1', \varphi_3') = (\omega(g) \varphi_1, \varphi_3) \cdot (\varphi_1^\perp, \varphi_3^\perp)
\]
for $g \in G$, we have
\[
 \int_G (\omega(g) \varphi_1, \varphi_3) \overline{(\hat{\pi}'(g) f_1', f_3')} \, dg \ne 0.
\]
Finally, by Lemma \ref{l:cont}, we may assume that $f_1'$ and $f_3'$ are $K$-finite vectors in $\HH$.
This shows that $\ZZ_{V,W,\chi_V,\chi_W,\psi}(\pi) \ne 0$ and completes the proof.
\end{proof}

Fix an $m$-dimensional Hermitian space $V$ over $\C$ and a character $\chi_V$ of $\C^\times$ such that $\chi_V|_{\R^\times} = \omega_{\C/\R}^m$.
Let $(r,s)$ be the signature of $V$.
Since $\omega_{V, W, \chi_V, \chi_W, \psi^{-1}} = \omega_{-V, W, \chi_V, \chi_W, \psi}$, we may assume that $\psi$ is as given in \S \ref{ss:main-theorem-setup} by replacing $V$ by $-V$ if necessary.
For any $n$-dimensional skew-Hermitian space $W$ over $\C$ and any discrete series representation $\pi$ of $\U(W)$ with Harish-Chandra parameter $\lambda$, we define integers $k_\lambda, r_\lambda, s_\lambda$ as in \S \ref{ss:nonvanishing} with respect to $k_0$ and $\chi_V$, where $k_0 = -1$ or $0$ is determined by
\[
 m \equiv n + k_0 \bmod 2. 
\]
Note that $k_\lambda \equiv k_0 \bmod 2$ and
\[
 r_\lambda + s_\lambda + k_\lambda = 
 \begin{cases}
  n-1 & \text{if $k_\lambda = -1$;} \\
  n & \text{if $k_\lambda \ge 0$.}
 \end{cases}
\]
As in Remark \ref{r:nonvanishing}, we may assume that $r-r_\lambda \ge s-s_\lambda$ by replacing $V$ by $-V$ and $\pi$ by $\bar{\pi} \otimes (\chi_V \circ \det)$ if necessary.
Then we have
\begin{equation}
\label{eq:(r,s)}
 (r,s) = 
 \begin{cases}
  (r_\lambda + l + 2t + 1, s_\lambda + l) & \text{if $k_\lambda = -1$;} \\
  (r_\lambda + l + 2t, s_\lambda + l) & \text{if $k_\lambda \ge 0$}
 \end{cases}
\end{equation}
for some integers $l,t$ with $t \ge 0$.
If further $\theta_{V,W,\chi_V,\chi_W,\psi}(\pi) \ne 0$, then it follows from Theorem \ref{t:nonvanishing} that 
\[
\begin{cases}
 l \ge 0, t \ge 0 & \text{if $k_\lambda = -1$;} \\
 l \ge 0, t \ge 0 & \text{if $k_\lambda = 0$;} \\
 \text{$l \ge 0, t = 0$ or $l \ge k_\lambda, t \ge 1$} & \text{if $k_\lambda \ge 1$.}
\end{cases}
\]

\begin{lem}
\label{l:induction}
Let $W$ be an $n$-dimensional skew-Hermitian space over $\C$ and $\pi$ a discrete series representation of $\U(W)$ with Harish-Chandra parameter $\lambda$.
Assume that $k_\lambda, r_\lambda, s_\lambda$ satisfy \eqref{eq:(r,s)} for some integers $l,t$ with
\[
\begin{cases}
 l \ge 0, t \ge 0 & \text{if $k_\lambda = -1$;} \\
 l \ge 0, t \ge 1 & \text{if $k_\lambda = 0$;} \\
 l \ge k_\lambda, t \ge 0 & \text{if $k_\lambda \ge 1$}
\end{cases}
\]
(and hence $m>n$).
Assume further that $\theta_{V,W,\chi_V,\chi_W,\psi}(\pi) \ne 0$.
Then there exist an $(n+1)$-dimensional skew-Hermitian space $W'$ over $\C$ containing $W$ and a discrete series representation $\pi'$ of $\U(W')$ with Harish-Chandra parameter $\lambda'$ such that
\begin{itemize}
\item $\theta_{V,W',\chi_V,\chi_{W'},\psi}(\pi') \ne 0$;
\item $\hat{\pi}$ occurs in $(\hat{\pi}'|_{\U(W)})_{\disc}$;
\item $k_{\lambda'}, r_{\lambda'}, s_{\lambda'}$ satisfy the following conditions:
\begin{itemize}
\item 
if $k_\lambda = -1$, then $k_{\lambda'} = 0$ and
\[
 (r,s) = (r_{\lambda'} + l + 2t, s_{\lambda'} + l);
\]
\item 
if $k_\lambda = 0$, then $k_{\lambda'} = -1$ and
\[
 \text{$(r,s) = (r_{\lambda'} + l + 2(t-1) + 1, s_{\lambda'} + l)$ or $(r_{\lambda'} + (l-1) + 2t + 1, s_{\lambda'} + (l-1))$,}
\]
where the second case happens only if $l \ge 1$;
\item
if $k_\lambda \ge 1$, then $k_{\lambda'} = k_\lambda-1$ and 
\[
 (r,s) = (r_{\lambda'} + (l-1) + 2t, s_{\lambda'} + (l-1)).
\]
\end{itemize}
\end{itemize}
\end{lem}

\begin{proof}
The assertion was essentially proved by Atobe (see Lemmas 5.1, 5.2, 5.3 of \cite{atobe} in the cases $k_\lambda = -1, k_\lambda = 0, k_\lambda \ge 1$, respectively). 
We only give some details in the case $k_\lambda = 0$.
Since $\theta_{V,W,\chi_V,\chi_W,\psi}(\pi) \ne 0$, we have $\# \CC_\lambda^\pm(l+t) \le l$ by Theorem \ref{t:nonvanishing}.
Let $W'$ be the skew-Hermitian space over $\C$ of signature $(p+1,q)$, where $(p,q)$ is the signature of $W$.
Then by \cite[Lemma 5.2]{atobe}, there exists a discrete series representation $\pi'$ of $\U(W')$ with Harish-Chandra parameter $\lambda'$ (relative to the choice of integers $l+t \ll \beta_0 < \beta_1 < \cdots$) satisfying the following conditions:
\begin{itemize}
\item $\Hom_{\U(W)}(\pi', \pi) \ne 0$;
\item $k_{\lambda'} = -1$;
\item $(r_{\lambda'}, s_{\lambda'}) =
\begin{cases}
 (r_\lambda + 1, s_\lambda) & \text{if $(\frac{1}{2}, +1)$ and $(-\frac{1}{2}, -1)$ do not belong to $\XX_\lambda$;} \\
 (r_\lambda, s_\lambda + 1) & \text{if $(\frac{1}{2}, +1)$ or $(-\frac{1}{2}, -1)$ belongs to $\XX_\lambda$;}
\end{cases}
$
\item $\# \CC^\pm_{\lambda'}(l+t-1) \le 
\begin{cases}
 l & \text{if $(\frac{1}{2}, +1)$ and $(-\frac{1}{2}, -1)$ do not belong to $\XX_\lambda$;} \\
 l-1 & \text{if $(\frac{1}{2}, +1)$ or $(-\frac{1}{2}, -1)$ belongs to $\XX_\lambda$.}
\end{cases}
$
\end{itemize}
We remark that Atobe defined $\pi'$ so that the pair $(\pi', \pi)$ satisfies the conditions on $\epsilon$-factors in the Gan--Gross--Prasad conjecture \cite{ggp1} and deduced from a result of He \cite{he} that $\hat{\pi}$ occurs in $(\hat{\pi}'|_{\U(W)})_{\disc}$ and hence $\Hom_{\U(W)}(\pi', \pi) \ne 0$.
Also, if $(\frac{\epsilon}{2}, \epsilon) \in \XX_\lambda$ for some $\epsilon = \pm 1$, then $(\frac{\epsilon}{2}, \epsilon) \in \XX_\lambda^{(\infty)}$.
Since $t \ge 1$, this implies that $\# \CC_\lambda^\epsilon(t) \ge 1$ and hence $l \ge 1$.
By Theorem \ref{t:nonvanishing} again, the conditions above imply that $\theta_{V,W',\chi_V,\chi_{W'},\psi}(\pi') \ne 0$.
This completes the proof.
\end{proof}

We now prove Proposition \ref{p:key}.
Let $W$ be an $n$-dimensional skew-Hermitian space over $\C$ and $\pi$ a discrete series representation of $\U(W)$ with Harish-Chandra parameter $\lambda$.
Assume that $m \ge n$ and that $k_\lambda, r_\lambda, s_\lambda$ satisfy \eqref{eq:(r,s)} for some integers $l,t$ with $t \ge 0$.
Assume further that $\theta_{V,W,\chi_V,\chi_W,\psi}(\pi) \ne 0$.
Then we need to show that $\ZZ_{V,W,\chi_V,\chi_W,\psi}(\pi) \ne 0$.

We first consider the case $k_\lambda = -1$ or $0$ (and hence $l \ge 0, t \ge 0$).
By Lemmas \ref{l:seesaw} and \ref{l:induction}, an induction on $l+t$ reduces us to the case $k_\lambda = 0, l \ge 0, t = 0$.
In this case, the assertion will be proved in Lemma \ref{l:base} below.

We next consider the case $k_\lambda \ge 1$.
If $t \ge 1$ (and hence $l \ge k_\lambda$), then by Lemmas \ref{l:seesaw} and \ref{l:induction}, an induction on $k_\lambda$ reduces us to the case $k_\lambda = 0, l \ge 0, t \ge 1$.
But this case has already been treated above.
If $t = 0$ (and hence $l \ge \frac{k_\lambda}{2}$), then the assertion will be proved in Lemma \ref{l:base} below.

\subsection{Base step}

We continue with the setup of the previous subsection.
To finish the proof of Proposition \ref{p:key}, it remains to prove the following.

\begin{lem}
\label{l:base}
Let $W$ be an $n$-dimensional skew-Hermitian space over $\C$ and $\pi$ a discrete series representation of $\U(W)$ with Harish-Chandra parameter $\lambda$.
Assume that $k_\lambda \ge 0$ and 
\[
 (r,s) = (r_\lambda + l, s_\lambda + l)
\]
for some integer $l \ge \frac{k_\lambda}{2}$ (and hence $m \ge n$).
Then $\ZZ_{V,W,\chi_V,\chi_W,\psi}(\pi) \ne 0$.
\end{lem}

Note that by Theorem \ref{t:nonvanishing}, the assumption automatically implies that $\theta_{V,W,\chi_V,\chi_W,\psi}(\pi) \ne 0$.

To prove this lemma, we need the notion of $K$-types of minimal degrees introduced by Howe \cite{howe2}.
Let $(p,q)$ be the signature of $W$.
We take the maximal compact subgroup $K \cong \U(p) \times \U(q)$ of $\U(W) = \U(p,q)$ as in \S \ref{ss:ds} and parametrize the irreducible representations of $K$ by highest weights
\[
 (a_1, \dots, a_p; b_1, \dots, b_q),
\]
where 
\begin{itemize}
\item $a_i, b_j \in \Z$;
\item $a_1 \ge \dots \ge a_p$ and $b_1 \ge \dots \ge b_q$.
\end{itemize}
Similarly, we take the maximal compact subgroup $K' \cong \U(r) \times \U(s)$ of $\U(V) = \U(r,s)$ and parametrize the irreducible representations of $K'$.

Let $\mathcal{P} = \bigoplus_{d=0}^\infty \mathcal{P}_d$ be the Fock model of the Weil representation $\omega_{V,W,\chi_V,\chi_W,\psi}$ of $\U(W) \times \U(V)$ relative to the data $(\chi_V,\chi_W,\psi)$ given in \S \ref{ss:main-theorem-setup}, where $\mathcal{P}$ is the space of polynomials in $mn$ variables and $\mathcal{P}_d$ is the subspace of homogeneous polynomials of degree $d$.
Note that $\mathcal{P}_d$ is invariant under the action of $K \times K'$.
For any irreducible representation $\mu$ of $K$ occurring in $\mathcal{P}$, we define the $(r,s)$-degree of $\mu$ as the smallest nonnegative integer $d$ such that the $\mu$-isotypic component of $\mathcal{P}_d$ is nonzero, which depends only on $r-s$ (see \cite[Lemma 1.4.5]{paul1}).

Let $\mathcal{H}$ be the space of joint harmonics, which is a $K \times K'$-invariant subspace of $\mathcal{P}$.
For any irreducible representations $\mu$ and $\mu'$ of $K$ and $K'$, respectively, we say that $\mu$ and $\mu'$ correspond if $\mu \boxtimes \mu'$ occurs in $\mathcal{H}$, in which case $\mu$ and $\mu'$ determine each other.
This correspondence can be described as follows.

\begin{lem}
\label{l:K-type-corresp}
Let $\mu$ and $\mu'$ be irreducible representations of $K$ and $K'$, respectively.
Then $\mu$ and $\mu'$ correspond if and only if $\mu$ and $\mu'$ are of the form
\begin{align*}
 \mu & = (a_1, \dots, a_x, 0, \dots, 0, b_1, \dots, b_y; c_1, \dots, c_z, 0, \dots, 0, d_1, \dots, d_w) \\
 & + \bigg( \frac{r-s}{2}, \dots, \frac{r-s}{2}; \frac{s-r}{2}, \dots, \frac{s-r}{2} \bigg) + \bigg( \frac{m_0}{2}, \dots, \frac{m_0}{2} \bigg)
\end{align*}
and 
\begin{align*}
 \mu' & = (a_1, \dots, a_x, 0, \dots, 0, d_1, \dots, d_w; c_1, \dots, c_z, 0, \dots, 0, b_1, \dots, b_y) \\
 & + \bigg( \frac{p-q}{2}, \dots, \frac{p-q}{2}; \frac{q-p}{2}, \dots, \frac{q-p}{2} \bigg) + \bigg( \frac{n_0}{2}, \dots, \frac{n_0}{2} \bigg),
\end{align*}
where
\begin{itemize}
\item $a_i, b_j, c_k, d_l \in \Z$;
\item $a_1 \ge \dots \ge a_x > 0 > b_1 \ge \dots \ge b_y$ and $c_1 \ge \dots \ge c_z > 0 > d_1 \ge \dots \ge d_w$;
\item $x+y \le p$ and $z+w \le q$;
\item $x+w \le r$ and $z+y \le s$.
\end{itemize}
\end{lem}

\begin{proof}
Given our choice of the data $(\chi_V, \chi_W, \psi)$, the assertion follows from \cite[Theorem 5.4]{konno}.
We remark that the convention in \cite{konno} is different from ours (see \cite[Lemma 3.1]{konno} and \cite[p.~758]{gi2}).
In particular, to switch the left and right actions of $\U(W)$ on $W$, we need to compose the Weil representation $\omega_{V,W, \underline{\xi}}$ as in \cite[\S 3.3]{konno} relative to the pair $\underline{\xi} = (\chi_W, \chi_V^{-1})$ with the automorphism $g \mapsto {}^t g^{-1}$ of $\U(p,q)$.
\end{proof}

Let $\pi$ be an irreducible representation of $\U(W)$ such that the theta lift $\theta_{V,W,\chi_V,\chi_W,\psi}(\pi)$ to $\U(V)$ is nonzero.
Let $\mu$ be a $K$-type of $\pi$, i.e.~an irreducible representation of $K$ occurring in $\pi|_K$.
We say that $\mu$ is of minimal $(r,s)$-degree in $\pi$ if the $(r,s)$-degree of $\mu$ is minimal among all $K$-types of $\pi$, in which case $\mu$ occurs in $\mathcal{H}$.

\begin{lem}
\label{l:lkt->mindeg}
Let $\pi$ be a discrete series representation of $\U(W)$ satisfying the assumption of Lemma \ref{l:base}.
Let $\mu$ be the lowest $K$-type of $\pi$.
Then $\mu$ is of minimal $(r,s)$-degree in $\pi$.
\end{lem}

\begin{proof}
Put $(r_0, s_0) = (r_\lambda + [\frac{k_\lambda}{2}], s_\lambda + [\frac{k_\lambda}{2}])$, so that $r_0+s_0 = n$ or $n-1$.
Let $V_0$ be the Hermitian space over $\C$ of signature $(r_0, s_0)$.
Then by Theorem \ref{t:nonvanishing}, the theta lift $\theta_{V_0,W,\chi_V,\chi_W,\psi}(\pi)$ to $\U(V_0)$ is nonzero.
Moreover, by \cite[Proposition 0.5]{paul1} and \cite[Proposition 1.4]{paul2}, $\mu$ is of minimal $(r_0,s_0)$-degree in $\pi$.
On the other hand, for any $K$-type $\nu$ of $\pi$, the $(r_0,s_0)$-degree of $\nu$ agrees with the $(r,s)$-degree of $\nu$.
Hence $\mu$ is of minimal $(r,s)$-degree in $\pi$.
\end{proof}

We also need a seesaw diagram
\[
\begin{tikzcd}
 \U(W) \times \U(W) \arrow[rd,dash] & \U(V) \arrow[ld,dash] \\
 \U(W) \arrow[u,dash] & \U(V_1) \times \U(V_2) \arrow[u,dash]
\end{tikzcd},
\]
where $V_1$ and $V_2$ are the Hermitian spaces over $\C$ of signatures $(r,0)$ and $(0,s)$, respectively, such that $V = V_1 \oplus V_2$ and $K' = \U(V_1) \times \U(V_2)$.
Consider the symplectic spaces
\[
 \W = V \otimes_E W, \quad 
 \W_1 = V_1 \otimes_E W, \quad
 \W_2 = V_2 \otimes_E W
\]
over $F$, so that $\W = \W_1 \oplus \W_2$.
We may take complete polarizations
\[
 \W = \X \oplus \Y, \quad
 \W_1 = \X_1 \oplus \Y, \quad
 \W_2 = \X_2 \oplus \Y
\]
such that $\X = \X_1 \oplus \X_2$ and $\Y = \Y_1 \oplus \Y_2$.
Write 
\[
 \omega = \omega_{V,W,\chi_V,\chi_W,\psi}, \quad
 \omega_1 = \omega_{V_1,W,\chi_{V_1},\chi_W,\psi}, \quad
 \omega_2 = \omega_{V_2,W,\chi_{V_2},\chi_W,\psi},
\]
where $\chi_{V_1}, \chi_{V_2}$ are characters of $\C^\times$ given by 
\[
 \chi_{V_1}(z) = \left( \frac{z}{\sqrt{z \bar{z}}} \right)^{m_1}, \quad
 \chi_{V_2}(z) = \left( \frac{z}{\sqrt{z \bar{z}}} \right)^{m_2} 
\]
with some integers $m_1, m_2$ such that
\[
 m_1 \equiv r \bmod 2, \quad
 m_2 \equiv s \bmod 2, \quad
 m_1 + m_2 = m_0.
\]
Then we have an identification
\[
 (\omega, S(\X)) = (\omega_1 \boxtimes \omega_2, S(\X_1) \otimes S(\X_2))
\]
as representations of $\U(W) \times \U(V_1) \times \U(V_2)$.

We now prove Lemma \ref{l:base}.
Let $\pi$ be a discrete series representation of $\U(W)$ satisfying the assumption of Lemma \ref{l:base}.
Let $\mu$ be the lowest $K$-type of $\pi$.
By Lemma \ref{l:lkt->mindeg}, we may write
\begin{align*}
 \mu & = (a_1, \dots, a_x, 0, \dots, 0, b_1, \dots, b_y; c_1, \dots, c_z, 0, \dots, 0, d_1, \dots, d_w) \\
 & + \bigg( \frac{r-s}{2}, \dots, \frac{r-s}{2}; \frac{s-r}{2}, \dots, \frac{s-r}{2} \bigg) + \bigg( \frac{m_0}{2}, \dots, \frac{m_0}{2} \bigg)
\end{align*}
as in Lemma \ref{l:K-type-corresp}.
Put
\begin{align*}
 \mu_1 & = (a_1, \dots, a_x, 0, \dots, 0; 0, \dots, 0, d_1, \dots, d_w)
 + \bigg( \frac{r}{2}, \dots, \frac{r}{2}; -\frac{r}{2}, \dots, -\frac{r}{2} \bigg) + \bigg( \frac{m_1}{2}, \dots, \frac{m_1}{2} \bigg), \\ 
 \mu_2 & = (0, \dots, 0, b_1, \dots, b_y; c_1, \dots, c_z, 0, \dots, 0)
 + \bigg( {-\frac{s}{2}}, \dots, -\frac{s}{2}; \frac{s}{2}, \dots, \frac{s}{2} \bigg) + \bigg( \frac{m_2}{2}, \dots, \frac{m_2}{2} \bigg),
\end{align*}
so that the tensor product representation $\mu_1 \otimes \mu_2$ contains $\mu$.
Let $\mu'$ be the irreducible representation of $K'$ corresponding to $\mu$.
Let $\mu_1'$ and $\mu_2'$ be the irreducible representations of $\U(V_1)$ and $\U(V_2)$, respectively, given by 
\begin{align*}
 \mu_1' & = (a_1, \dots, a_x, 0, \dots, 0, d_1, \dots, d_w)
 + \bigg( \frac{p-q}{2}, \dots, \frac{p-q}{2} \bigg) + \bigg( \frac{n_0}{2}, \dots, \frac{n_0}{2} \bigg), \\
 \mu_2' & = (c_1, \dots, c_z, 0, \dots, 0, b_1, \dots, b_y)
 + \bigg( \frac{q-p}{2}, \dots, \frac{q-p}{2} \bigg) + \bigg( \frac{n_0}{2}, \dots, \frac{n_0}{2} \bigg),
\end{align*}
so that $\mu' = \mu_1' \boxtimes \mu_2'$.
Then the theta lift $\pi_i = \theta_{V_i,W,\chi_{V_i},\chi_W, \psi}(\mu_i')$ to $\U(W)$ is nonzero.
In fact, $\pi_i$ is the unitary highest weight module with lowest $K$-type $\mu_i$ (see \cite{kashiwara-vergne}).
Since $\U(V_i)$ is compact, we may realize the representation $\pi_i \boxtimes \mu_i'$ of $\U(W) \times \U(V_i)$ on the $\mu_i'$-isotypic component $S(\X_i)_{\mu_i'}$ of $S(\X_i)$.
In particular, for $\varphi_{1,i}, \varphi_{2,i} \in S(\X_i)_{\mu_i'}$, the function 
\[
 g \longmapsto (\omega_i(g) \varphi_{1,i}, \varphi_{2,i})
\]
is a matrix coefficient of $\pi_i$.

Thus it remains to show that the integral
\begin{equation}
\label{eq:matcoeff-integral}
 \int_{\U(W)} \Psi_1(g) \Psi_2(g) \overline{\Psi(g)} \, dg 
\end{equation}
is nonzero for some linear combinations $\Psi_1, \Psi_2, \Psi$ of matrix coefficients of $\pi_1, \pi_2, \pi$, respectively.
Indeed, the integral \eqref{eq:matcoeff-integral} is a linear combination of integrals of the form 
\[
 \int_{\U(W)} (\omega(g) \varphi_1, \varphi_2) \overline{(\pi(g) f_1, f_2)} \, dg,
\]
where $\varphi_1 = \varphi_{1,1} \otimes \varphi_{1,2}, \varphi_2 = \varphi_{2,1} \otimes \varphi_{2,2} \in S(\X)$ with $\varphi_{1,i}, \varphi_{2,i} \in S(\X_i)_{\mu_i'}$ and $f_1, f_2 \in \pi$.
We take the Flensted-Jensen function $\Psi$ given by 
\[
 \Psi(g) = \frac{1}{\dim \mu} \cdot \Tr(P_\mu \pi(g) P_\mu),
\]
where $P_\mu$ is the orthogonal projection to the $\mu$-isotypic component of $\pi$ (see \cite[\S 7]{flensted-jensen}).
Similarly, we take the function $\Psi_i$ given by 
\[
 \Psi_i(g) = \frac{1}{\dim \mu_i} \cdot \Tr(P_{\mu_i} \pi_i(g) P_{\mu_i}).
\]
Then it follows from the proof of \cite[Theorem 4.1]{li90} that the integral \eqref{eq:matcoeff-integral} is nonzero.
This completes the proof of Lemma \ref{l:base} and hence of Proposition \ref{p:key}.

This also completes the proof of Theorem \ref{t:main}.

\end{document}